\documentclass[reqno,a4paper]{amsart}\usepackage[english]{babel}
\usepackage[T1]{fontenc}
\usepackage[utf8]{inputenc}
\usepackage{lmodern}
\usepackage[left=3cm,right=3cm,top=3cm,bottom=3cm]{geometry} 
\usepackage{mathtools}
\usepackage{amssymb}
\usepackage{amsmath}
\usepackage{amsthm}
\usepackage{tikz}
\usepackage{tikz-cd}
\usepackage{cases}
\usepackage{enumerate}
\usepackage{mathabx,epsfig}
\usepackage{mathrsfs}
\usepackage{setspace}
\usepackage{csquotes}
\usepackage[bookmarksopen,bookmarksdepth=1]{hyperref}
\usepackage{thmtools}
\usepackage{thm-restate}
\usepackage{cleveref}
\usepackage{epstopdf}
\usepackage{wasysym}

\usepackage[style=alphabetic,giveninits=true,backend=bibtex,isbn=false,url=false, doi=false, maxbibnames=99,maxalphanames=4]{biblatex}
\bibliography{references}

\theoremstyle{plain}
  \newtheorem{thm}{Theorem}[section]
  \newtheorem{theorem}[thm]{Theorem}

  \newtheorem{corollary}[thm]{Corollary}
    \newtheorem{proposition}[thm]{Proposition}
  \newtheorem{lemma}[thm]{Lemma}
   \newtheorem{conjecture}[thm]{Conjecture}

\theoremstyle{definition}
  \newtheorem{definition}[thm]{Definition}
  \newtheorem{example}[thm]{Example}

  \newtheorem{remark}[thm]{Remark}
  
  \newtheorem{convention}[thm]{Convention}

\newcommand{\bZ}{\mathbb{Z}} 
\newcommand{\bQ}{\mathbb{Q}}

\newcommand{\bC}{\mathbb{C}}

\newcommand{\cC}{\mathcal{C}}
\newcommand{\cA}{\mathcal{A}}

\newcommand{\cK}{\mathcal{K}}
\newcommand{\cH}{\mathcal{H}}
\newcommand{\cG}{\mathcal{G}}

\newcommand{\PP}{\mathcal{P}}
\newcommand{\alg}{\textrm{alg}}
\newcommand{\sgn}{\textrm{sgn}}
\newcommand{\dashmod}{\text{-mod}}
\newcommand{\id}{\mathrm{id}}
\newcommand{\rM}{\mathrm{M}}
\newcommand{\ch}{\mathrm{ch}}
\newcommand{\gen}{\mathrm{gen}}
\newcommand{\pres}{\mathrm{pres}}
\newcommand{\ab}{\mathrm{ab}}

\newcommand\Com{\mathcal{C}om}
\newcommand\calP{\mathcal{P}}
\newcommand\calC{\mathcal{C}}
\newcommand\calO{\mathcal{O}}

\makeatletter
\newcommand{\exterior}[1]{\mathop{\mathpalette\exterior@{#1}}}
\newcommand{\exterior@}[2]{%
  \raisebox{\depth}{%
  \fontsize{\sf@size}{0}%
  \m@th
  $\ifx#1\displaystyle\textstyle\else#1\fi\bigwedge$}%
  ^{\mspace{-2mu}#2}%
  \kern-\scriptspace
}

\newcommand{\sectionauthor}[1]{%
  {\linespread{1.1}\scshape#1%
  \par\nobreak}
  \@afterheading%
}
\makeatother

\definecolor{dOrange}{rgb}{0.9,0.4,0}
\newcommand{\uncertain}[1]{\textcolor{dOrange}{#1}}

\newcommand{\nocontentsline}[3]{}
\newcommand{\tocless}[2]{\bgroup\let\addcontentsline=\nocontentsline#1{#2}\egroup}

\DeclareMathOperator{\Aut}{Aut}
\DeclareMathOperator{\GL}{GL}

\DeclareMathOperator{\IA}{IA}

\DeclareMathOperator{\FB}{\mathsf{FB}}

\DeclareMathOperator{\Rep}{Rep}
\DeclareMathOperator{\Hom}{Hom}

\DeclareMathOperator{\wBr}{\mathsf{wBr}}
\DeclareMathOperator{\dwBr}{\mathsf{dwBr}}

\DeclareMathOperator{\tensoredover}{\hspace{0.2mm}\astrosun\hspace{0.2mm} }
\DeclareMathOperator{\GGL}{\mathbf{GL}}
\DeclareMathOperator{\SL}{SL}
\DeclareMathOperator{\Gr}{Gr}

\DeclareMathOperator{\coker}{coker}

\title{The walled Brauer category and stable cohomology of $\IA_n$}
\author{Erik Lindell, \\ With an appendix by Mai Katada}
\date{}

\begin{document}

\maketitle

\begin{abstract}
    The $\IA$-automorphism group is the group of automorphisms of the free group $F_n$ that act trivially on the abelianization $F_n^{\ab}$. This group is in many ways analoguous to Torelli groups of surfaces and their higher dimensional analogues. In recent work, the stable rational cohomology of such groups was studied by Kupers and Randal-Williams, using the machinery of so-called Brauer categories. In this paper, we adapt their methods to study the stable rational cohomology of the $\IA$-automorphism group. We obtain a conjectural description of the algebraic part of the stable rational cohomology and prove that it holds up to degree $Q+1$, given the assumption that the stable cohomology groups are stably finite dimensional in degrees up to $Q$. In particular, this allows us to compute the algebraic part of the stable cohomology in degree 2, which we show agrees with the part generated by the first cohomology group via the cup product map and which has previously been computed by Pettet.

    In the appendix, written by Mai Katada, it is shown how the results of the paper can be applied to compute the stable Albanese (co)homology of the $\IA$-automorphism group.
\end{abstract}

\setcounter{tocdepth}{1}

\tableofcontents

\section{Introduction}
\noindent Let $F_n:=\langle x_1,x_2,\ldots,x_n\rangle$ denote the free group on $n$ generators and $\Aut(F_n)$ its automorphism group. The action on $F_n^{\ab}\cong\bZ^n$ defines a surjective homomorphism $\Aut(F_n)\to\GL_n(\bZ)$ and we define $\IA_n$, the $\IA$-automorphism group of $F_n$, as the kernel of this action. The subject of this paper is the rational cohomology of this group. Similarly to the groups $\Aut(F_n)$, the $\IA$-automorphism groups fit into a sequence
\begin{equation}\label{eq:stabilization}
    \IA_1\to\IA_2\to\cdots\to\IA_n\to\IA_{n+1}\to\cdots
\end{equation}
given by extending an automorphism by acting trivially on the new generator. However, unlike for the corresponding sequence of the groups $\Aut(F_n)$, the induced sequence in (co)homology of $\IA_n$ does not stabilize. On the other hand, by definition $\IA_n$ sits in a short exact sequence
\begin{equation}\label{eq:IA-ses}
    1\to\IA_n\to\Aut(F_n)\to\GL_n(\bZ)\to1,
\end{equation}
making $H^*(\IA_n,\bQ)$ into a $\GL_n(\bZ)$-representation. When taking this extra structure into account, the cohomology groups still turn out to be the most tractable in a range where $n$ is sufficiently large compared to the cohomological degree. For example, it was proven by Kawazumi \cite{KawazumiMagnusExpansions}, and independently by Cohen-Pakianathan and Farb (both unpublished), that we have
$$H^1(\IA_n,\bQ)\cong \exterior{2} H^\vee(n)\otimes H(n),$$
where $H(n):=H_1(F_n,\bQ)$ and $H^\vee(n):=H^1(F_n,\bQ)$. Based on this, it was conjectured by Church and Farb \cite{ChurchFarbRepStability} that the sequence induced in (co)homology by (\ref{eq:stabilization}) satisfies \textit{representation stability}. Since we will not consider representation stability in any detail in this paper, we refer to \cite{ChurchFarbRepStability} for the precise definition, and only say that in this situation it essentially says that for $n$ sufficiently large compared to the degree $*$, $H^*(\IA_n,\bQ)$ is the restriction of a $\GL_n(\bQ)$-representation (i.e.\ an \textit{algebraic} representation of $\GL_n(\bZ)$) and that its decomposition into irreducible $\GL_n(\bQ)$-representations is independent of $n$.

The stable rational cohomology of $\IA_n$ has also recently been studied by Katada \cite{Katada2022stable} and Habiro-Katada \cite{HabiroKatada2023stable}. In \cite{Katada2022stable}, Katada studied that so called \textit{Albanese cohomology} of $\IA_n$, which is defined as the image of the cup product map $\exterior{*}H^1(\IA_n,\bQ)\to H^*(\IA_n,\bQ)$ and denoted $H_A^*(\IA_n,\bQ)$. Based on her calculations, she conjectured a complete description of the stable Albanese cohomology, which can be slightly reformulated as follows:

\begin{conjecture}{\cite[Conjecture 8.2]{Katada2022stable}}\label{conjecture:Katada}
    For $n\gg *$, we have
    $$H^*_A(\IA_n,\bQ)\cong \exterior{*}(\exterior{2}H^\vee(n)\otimes H(n))/\mathsf{IH},$$
where $\mathsf{IH}$ is the ideal generated by elements of the form
\begin{align}\label{eq:relationmainthm}
        \sum_{i=1}^n (f_1\wedge f_2\otimes e_i)\wedge(e_i^\#\wedge f_3\otimes v)-(f_3\wedge f_1\otimes e_i)\wedge(e_i^\#\wedge f_2\otimes v)\in \exterior{2}(\exterior{2}H^\vee(n)\otimes H(n)),
    \end{align}
where $\{e_1,\ldots,e_n\}$ is the standard basis of $H(n)$, defined by $e_i:=[x_i]$, and $\{e_1^\#,\ldots,e_1^\#\}$ is the dual basis. 
\end{conjecture}

\begin{remark}
    We use the notation $\mathsf{IH}$ for this ideal as the relation (\ref{eq:relationmainthm}) corresponds to a ``directed IH-relation'' among $(2,1)$-valent directed graphs, as explained in Section \ref{sec:ring-structure} below.
\end{remark}

In the same paper, Katada introduced a graded $\GL_n(\bZ)$-representation $W_*(n)$ (see Appendix \ref{Appendix} for a definition of this representation) and proved that for $n\gg *$, $W_*(n)^\vee\subseteq H_A^*(\IA_n,\bQ)$ \cite[Theorem 6.1]{Katada2022stable}. Based on low degree cases, she conjectured that in fact 
$$W_*(n)^\vee\cong H^*_A(\IA_n,\bQ)\cong\exterior{*}(\exterior{2}H^\vee(n)\otimes H(n))/\mathsf{IH} $$
for $n\gg *$ (\cite[Conjecture 1.7]{Katada2022stable}).

In \cite{HabiroKatada2023stable}, the authors study the entire stable rational cohomology of $\IA_n$ and conjecture a complete description, based on Katada's conjecture and a weaker version of the conjecture by Church and Farb that $H^*(\IA_n,\bQ)$ satisfies representation stably, namely that it is stably a finite dimensional and algebraic representation of $\GL_n(\bZ)$. 

One of the main inputs for their calculation is the stable cohomology of $\Aut(F_n)$ with the coefficients $H(n)^{\otimes p}\otimes H^\vee(n)^{\otimes q}$, which was recently computed by the author in \cite{Lindell2022stable}. By studying the Hochschild-Serre spectral sequence associated to the short exact sequence (\ref{eq:IA-ses}), with these representations as coefficients, they proved the following:

\begin{theorem}{cf. \cite[Conjecture 7.5, Theorem 7.6]{HabiroKatada2023stable}}
    Assuming the conjecture that stably $H_A^*(\IA_n,\bQ)\cong W_*(n)^\vee$ and that stably $H^*(\IA_n,\bQ)$ is an algebraic $\GL_n(\bZ)$-representation in each degree, it follows that in a range $n\gg *$ we have a $\GL_n(\bZ)$-equivariant ring isomorphism
    $$H^*(\IA_n,\bQ)\cong \exterior{*}(\exterior{2}H^\vee(n)\otimes H(n))/\mathsf{IH}\otimes \bQ[y_4,y_8,\ldots],$$
    where $\bQ[y_4,y_8,\ldots]$ is a trivial graded $\GL_n(\bZ)$-representation with $|y_{4i}|=4i$.
\end{theorem}

\begin{remark}
    We recall that the stable rational cohomology of $\GL_n(\bZ)$ was computed by Borel in \cite{Borel1} and is isomorphic to the exterior algebra $\exterior{*}\{x_5,x_9,\ldots\}$, where $|x_{4i+1}|=4i+1$. The trivial graded representation $\bQ[y_4,y_8,\ldots]$ in the theorem above appears when studying so-called ``anti-transgression maps'' in the Hochschild-Serre spectral sequence associated to the short exact sequence (\ref{eq:IA-ses}), which gives rise to a degree $-1$ map $\exterior{*}\{x_5,x_9,\ldots\}\to H^*(\IA_n,\bQ)^{\GL_n(\bZ)}$. We explain this in more detail in Section \ref{subsec:invariantstrivialcoeffs} (cf. \cite[Section 6]{HabiroKatada2023stable}).
\end{remark}

The idea of studying the stable cohomology of $\IA_n$ using this spectral sequence is not unique to the work of Habiro and Katada. For example, in \cite{KupersRandal-Williams} and \cite{RW-TorelliII}, Kupers-Randal-Williams and Randal-Williams used this kind of approach to study the stable cohomology of Torelli groups of surfaces and their higher dimensional analogues. Furthermore, by utilizing a category theoretical approach using so-called \textit{Brauer categories}, Kupers and Randal-Williams were able to  effectively study the structure of these cohomology groups simoultaneously as representations of arithmetic groups as well as rings. The goal of this paper is to apply the methods of \cite{KupersRandal-Williams} to the study of $\IA_n$. This will allow us to retrieve the results of Habiro and Katada in a somewhat different way and also to push them a bit further.

To state our main result, we must first recall a little bit more background. Although not phrased directly like this in their paper, the crucial part in the assumption of Habiro and Katada that $H^*(\IA_n,\bQ)$ is stably finite dimensional and algebraic is the implication that $H^*(\IA_n,\bQ)$ satisfies \textit{Borel vanishing}:

\begin{definition}
    A $\GL_n(\bZ)$-representation $V$ is said to satisfy \textit{Borel vanishing} if in some range $n\gg*$, the cup product map
    $$H^*(\GL_n(\bZ),\bQ)\otimes V^{\GL_n(\bZ)}\to H^*(\GL_n(\bZ),V)$$    is an isomorphism.
\end{definition}

\begin{remark}
    Most ``reasonable'' representations of $\GL_n(\bZ)$ satisfy this. For example, Borel proved in \cite{Borel1} that any finite dimensional algebraic representation does and in Section \ref{sec:representationtheory} we show that this actually implies Borel vanishing for any finite dimensional $\GL_n(\bZ)$-representation. 
\end{remark}

This means that we do not need to make any assumptions about the algebraicity of the representation $H^*(\IA_n,\bQ)$. However, the method we employ is still only able to detect the ``algebraic part'' of the cohomology. For $V$ a $\GL_n(\bZ)$-representation we define its \textit{algebraic part} $V^{\alg}$ as
$$V^{\alg}:=\bigcup_{\substack{W\subseteq V\\ W\text{ algebraic}}}W\subseteq V.$$
With these definitions, our main theorem can be stated as follows:

\begin{restatable}{thmx}{theoremA}\label{theoremA}
There is a $\GL_n(\bZ)$-equivariant graded ring homomorphism
    $$\exterior{*}\left(\exterior{2}H^\vee(n)\otimes H(n)\right)\otimes\bQ[y_4,y_8,\ldots]\to H^*(\IA_n,\bQ)^{\alg}$$
    where $\bQ[y_4,y_8,\ldots]$ is a trivial graded $\GL_n(\bZ)$-representation. If there is a $Q\ge 0$ such that $H^*(\IA_n,\bQ)$ satisfies Borel vanishing in each degree $*<Q$, then in a range $n\gg *$, this ring homomorphism is surjective in degrees $*\le Q$ and the kernel is the ideal $\mathsf{IH}$.  
\end{restatable}

We note that, as expected, we obtain the same conjectural description of $H^*(\IA_n,\bQ)^{\alg}$ as Habiro and Katada (since in their case they assume that the entire stable cohomology is algebraic). However, we remark that in this theorem, there is no assumption about the Albanese cohomology of $\IA_n$.

Let us also remark on the degree shift appearing in the theorem and its corollary, which is obtained by pushing the spectral sequence argument of Habiro and Katada a bit further (see Section \ref{sec:catdescriptionofcohomology} below). Since we have seen that $H^1(\IA_n,\bQ)\cong\exterior{2}H^\vee(n)\otimes H(n)$ is in fact a finite dimensional (even algebraic) representation of $\GL_n(\bZ)$, and thus satisfies Borel vanishing, we get in particular the following corollary:

\begin{corollary}
    Let $\mathsf{IH}^2$ denote the degree 2 subspace of the ideal $\mathsf{IH}$. For $n\gg 2$, we have 
    $$H^2(\IA_n,\bQ)^{\alg}=H_A^2(\IA_n,\bQ)\cong \exterior{2}\left(\exterior{2} H^\vee(n)\otimes H(n)\right)/\mathsf{IH}^2.$$
\end{corollary}

\begin{remark}
    The stable second rational Albanese cohomology of $\IA_n$ has previously been computed by Pettet \cite{Pettet} and this corollary shows that this is actually all of the stable rational algebraic cohomology of $\IA_n$ in degree 2.
\end{remark}

\subsection{Categorical version of the main theorem} Theorem \ref{theoremA} is derived from a categorical description, which we derive using the methods of \cite{KupersRandal-Williams}, applied to the so-called \textit{downward walled Brauer category}. Let us briefly summarize the definition of this category, so that we can state the categorical version of the main result in these terms. To make the notation a bit less cluttered, let us write $G:=\GL_n(\bZ)$ for the rest of this subsection.

\begin{figure}[h]
    \centering
    \includegraphics[scale=0.55]{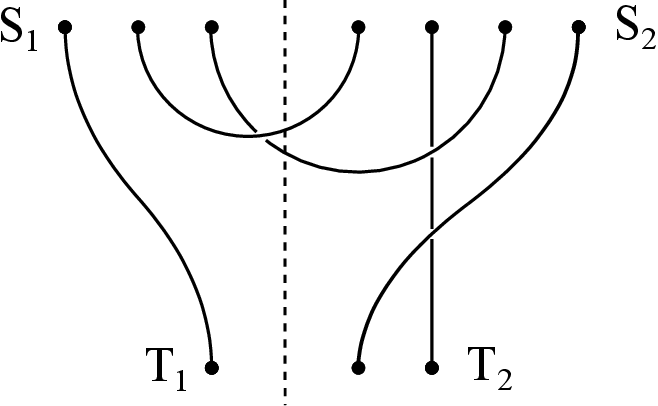}
    \caption{A morphism in $\dwBr$.}
    \label{fig:dwBr}
\end{figure}

\begin{definition}
    The downward walled Brauer category, which we denote by $\dwBr$, is the $\bQ$-linear category whose objects are pairs of finite sets $S=(S_1,S_2)$ and where the space of morphisms $\Hom_{\dwBr}(S,T)$ is spanned by diagrams such as those in Figure \ref{fig:dwBr} (where the order of crossings is not part of the data, but only included in the image to make it easier to read). Composition is given by concatenation of such diagrams. 
\end{definition}

We define a functor $K:\dwBr\to\Gr(\Rep(G))$ (the category of graded $G$-representations) by $(S_1,S_2)=H(n)^{\otimes S_1}\otimes H^\vee(n)^{\otimes S_2}$, which we think of as a graded representation concentrated in degree 0, and which on morphisms is given by applying the duality pairing $H^\vee(n)\otimes H(n)\to\bQ$ to the tensor factors that are matched via horizontal strands in a diagram like the one in Figure \ref{fig:dwBr}, and permuting the remaining ones according to the vertical strands. 

By taking linear duals, we get a functor $K^\vee:\dwBr^{\mathrm{op}}\to\Gr(\Rep(G))$, which we use to construct a functor $K^\vee\otimes^{\dwBr}(-)$ from the category of functors $\dwBr\to\Gr(\Rep(G))$ to $\Gr(\Rep(G))$, using the coend
$$K^\vee\otimes^{\dwBr} M:=\int^{x\in\dwBr}K^\vee(x)\otimes M(x).$$
The category $\dwBr$ has a symmetric monoidality given by disjoint union and this allows us to define a symmetric monoidality on the category of functors $\dwBr\to\Gr(\Rep(G))$ by Day convolution. It follows from the general results of \cite[Section 2]{KupersRandal-Williams} that the functor $K^\vee\otimes^{\dwBr}(-) $ is strongly symmetric monoidal.

We will use this functor to describe $H^*(\IA_n,\bQ)^\alg$, so let us briefly describe the functor which will play the role of $M$. We define a functor $\PP':\dwBr\to\Gr(\bQ\dashmod)$, where $\Gr(\bQ\dashmod)$ is the category of graded $\bQ$-vector spaces (which we may also consider as the subcategory of $\Gr(\Rep(G))$ consisting of trivial graded representations), by letting $\PP'(S)$ be the graded vector space with a basis given by partitions of the set $S_1$ with\begin{enumerate}
    \item at least $|S_2|$ parts,
    \item  $|S_2|$ of the parts are labeled by the elements of $S_2$ and the remaining parts are unlabeled,
    \item and where there are no labeled parts of size 1.
\end{enumerate}
We consider such a labeled partition as having degree $|S_1|-|S_2|$, so in particular $\PP'(S)$ is concentrated in a single degree.

For a morphism described by a diagram such as that in Figure \ref{fig:dwBr}, the functor acts by permuting and relabeling according to the vertical strands, and merging parts according to the horizontal strands. We describe this in detail in Section \ref{subsec:Functoriality} below.

Tensoring pointwise with the sign representation $\det(S):=\det\left(\bQ^{S_1}\right)\otimes\det\left(\bQ^{S_1}\right)$ of the product of the symmetric groups on $S_1$ and $S_2$, which we consider as a graded representation concentrated in degree 0, we obtain a functor $\PP'\otimes\det:\dwBr\to\Gr(\bQ\dashmod)$.
We prove in Section \ref{sec:twistedcoeffs} that $\PP'\otimes\det$ is a commutative ring object in the category of functors from $\dwBr$ to $\bQ\dashmod$, so $K^\vee\otimes^{\dwBr}(\PP'\otimes\det)$ is a commutative ring object in $\Gr(\Rep(G))$. With this, we can now state the categorical version of the main theorem:

\begin{restatable}{thmx}{theoremB}\label{theoremB}
  There is a graded ring homomorphism
    \begin{equation}\label{eq:thmB}
        K^\vee\otimes^{\dwBr}(\PP'\otimes\det)\otimes\bQ[y_4,y_8,\ldots]\to H^*(\IA_n,\bQ)^{\alg}
    \end{equation}
     If there is a $Q\ge 0$ such that $H^*(\IA_n,\bQ)$ satisfies Borel vanishing in each degree $*<Q$, then in a range $n\gg *$, the ring homomorphism (\ref{eq:thmB}) is an isomorphism in degrees $*\le Q$.
\end{restatable}

Theorem \ref{theoremA} is proven from Theorem \ref{theoremB} by showing that the graded commutative ring $K^\vee\otimes^{\dwBr}(\PP'\otimes\det)$ is in fact isomorphic to $\exterior{*}(\exterior{2}H^\vee(n)\otimes H(n))/\mathsf{IH}$, in a stable range of degrees. This is proven in Section \ref{sec:ring-structure} below.

\begin{remark}
    In correspondence with the author, regarding an earlier draft of this paper, it was explained by Mai Katada that in a stable range of degrees we have that $K^\vee\otimes^{\dwBr}(\PP'\otimes\det)\cong W_*(n)^\vee$, as graded $\GL_n(\bZ)$-representations, and that combining this fact with certain results from \cite{Katada2022stable}, it is possible to prove her conjecture about the stable Albanese (co)homology. She kindly agreed to add her proof in Appendix \ref{Appendix} below.
\end{remark}

\subsection{Overview of the paper} Let us give a brief overview of the structure of the paper. 

In Section \ref{sec:representationtheory}, we recall some background on representation theory, including Schur-Weyl duality and the invariant theory of the general linear group. After this we cover Borel vanishing in some more detail and show that every representation of $\GL_n(\bZ)$ which is finite dimensional satisfies Borel vanishing. 

In Section \ref{sec:walledBrauer} we start by recalling the general categorical setup from \cite{KupersRandal-Williams} and some of their general results. We then apply this to the walled Brauer category (a category which has $\dwBr$ as a subcategory) and prove the main categorical lemmas concerning this category needed to prove Theorem \ref{theoremB}. This section closely follows \cite[Sections 2.2-2.3]{KupersRandal-Williams}.

In Section \ref{sec:twistedcoeffs}, we recall the main result of \cite{Lindell2022stable}, which is a description of the stable cohomology of $\Aut(F_n)$ with coefficients in the representations $H(n)^{\otimes p}\otimes H^\vee(n)^{\otimes q}$, in terms of a space spanned by partitions of $\{1,\ldots,p\}$ labeled by the set $\{1,\ldots,q\}$, for any $p,q\ge 0$. However, in \cite{Lindell2022stable} only the existence of an isomorphism between these spaces is demonstrated, so in this section we also construct such an isomorphism explicitly, using the results of Kawazumi \cite{KawazumiMagnusExpansions} and Kawazumi-Vespa \cite{KawazumiVespa}. After this, we prove that this actually defines an isomorphism of symmetric monoidal functors on the walled Brauer category.

In Section \ref{sec:catdescriptionofcohomology}, we prove Theorem \ref{theoremB}, by studying the Hochschild-Serre spectral sequence associated to the short exact sequence (\ref{eq:IA-ses}) and then combining this with the categorical tools from the previous sections. The analysis of the spectral sequence closely follows the methods of \cite{HabiroKatada2023stable} and applies some of their results directly, whereas some parts requires some modifications in order to obtain the degree shift that appears in Theorem \ref{theoremB}.

In Section \ref{sec:ring-structure}, we then apply the methods of \cite[Section 5]{KupersRandal-Williams} to describe the ring structure of $K^\vee\otimes^{\dwBr}(\PP'\otimes\det)$ in terms of generators and relations and thereby obtain Theorem \ref{theoremA} from Theorem \ref{theoremB}. 

In Appendix \ref{Appendix}, written by Mai Katada, it is explained how to use the results of Section \ref{sec:ring-structure} to compute the stable Albanese (co)homology of $\IA_n$.

\subsection{Conventions}\label{subsec:conventions} In the entire paper, we will work over $\bQ$. In particular, all cohomology is taken with $\bQ$-coefficients, unless stated otherwise. We will write $\bQ\dashmod$ for the category of $\bQ$-vector spaces. If $G$ is a group, we write $\Rep(G)$ for the category of $G$-representations over $\bQ$. For a set $A$, we write $\bQ A$ for the vector space with $A$ as basis.

Throughout the paper, we consider several categories where the objects are pairs of finite sets. If $S=(S_1,S_2)$ is such a pair and $I=(I_1,I_2)$ is another, then for brevity we will write $I\subseteq S$ whenever $I_1\subseteq S_1$ and $I_2\subseteq S_2$. If $I\subseteq S$, we will also write $S\setminus I:=(S_1\setminus I_1,S_2\setminus I_2)$ and if $S=(S_1,S_2)$ and $T=(T_1,T_2)$ are two pairs of finite sets, we will write $S\sqcup T:=(S_1\sqcup T_1,S_2\sqcup T_2)$. Furthermore, we write $|S|:=|S_1|+|S_2|$ for the sum of the cardinalities of the finite sets in the pair $S$. 

For any positive integer $k\ge 0$, we will use the notation $[k]:=\{1,2,\ldots,k\}$ for $k\ge 1$ and $[k]=\varnothing$ for $k=0$. If $\Lambda$ is a category whose objects are pairs of fintie sets and $F:\Lambda\to\mathcal{C}$ is a functor to some category $\cC$, we will typically write $F(p,q):=F([p],[q])$ for brevity.

\subsection*{Acknowledgements} This project started while EL was a PhD student at Stockholm University, and he would like to thank his PhD advisor Dan Petersen for many useful discussions and helpful comments on the paper. He would also like to thank Alexander Kupers and Oscar Randal-Williams for answering numerous questions about their paper \cite{KupersRandal-Williams}. Furthermore, he would like to thank Mai Katada for writing the appendix to this paper, as well as for many helpful comments. He is also grateful to Najib Idrissi for many constructive discussions and comments on the paper. Finally, he would like to thank Kazuo Habiro and Arthur Soulié for helpful comments on earlier drafts of the paper.

MK would like to thank Erik Lindell for sending his draft to her.
Proposition 6.2 motivated her to prove her conjecture on the Albanese (co)homology of $\IA_n$.
She also thanks Kazuo Habiro for valuable discussions at the University of Tokyo. 
She was supported by JSPS KAKENHI Grant Number 
JP22KJ1864.

\section{Background on representation theory}\label{sec:representationtheory}

\noindent In this section, we recall some background on representation theory needed to prove the main theorem. In order to study Torelli groups of surfaces and their higher dimensional analogues in \cite{KupersRandal-Williams}, Kupers and Randal-Williams use representation theory of orthogonal and symplectic groups. In our setting, we instead need to consider representations of the general linear group. In this section we therefore retread much of what Kupers and Randal-Williams cover in Section 2 of their paper, but in the case of the general linear group.

\subsection{Representation theory of the general linear group} Let $H(n):=H_1(F_n,\bQ)$. Let $\GL_n(\bQ)$ denote the group of $\bQ$-linear automorphisms of $H(n)$.  The group $\GL_n(\bQ)$ is the $\bQ$-points of the algebraic group $\GGL_n$, which is defined over $\bQ$.

\subsubsection{Irreducible representations of $\GGL_n$ and of symmetric groups} We start by recalling the definition of the irreducible representations of the algebraic group $\GGL_n$, as well as those of symmetric groups, as these are intimately connected. In particular, the irreducible representations of both groups are classified using \textit{partitions}.

 \begin{definition}
     A parition is a decreasing sequence $\lambda=(\lambda_1\ge\lambda_2\ge\cdots\ge\lambda_l\ge 0)$ of non-negative integers that eventually reaches zero. We call $|\lambda|:=\lambda_1+\cdots+\lambda_l$ the \textit{weight} of the partition, and if $|\lambda|=k$ for some non-negative integer $k$, we say that $\lambda$ is a partition of $k$. The largest $l$ such that $\lambda_l>0$ is called the \textit{length} of the partition and we denote it by $l(\lambda)$. 
 \end{definition}

The irreducible representations of the symmetric groups $\Sigma_k$ are called \textit{Specht modules} and are indexed by partitions of $k$. We denote the Specht module associated to a partition $\lambda$ of weight $k$ by $S^\lambda$. This is the image of the \textit{Young symmetrizer} $c_\lambda\in\bQ[\Sigma_k]$ acting on $\bQ[\Sigma_k]$ (see for example \cite{FultonHarris} for a definition of $c_\lambda$).

We can also associate a representation of $\GGL_n$ to any partition $\lambda$, of any weight $k$, by applying the \textit{Schur functor} associated to $\lambda$:
\begin{equation}
V_\lambda:=S_\lambda(H(n))=S^\lambda\otimes_{\Sigma_k} H(n)^{\otimes k}.    
\end{equation}
This representation is always irreducible, and non-zero as long as the length of $\lambda$ is at most $n$. For example, for $\lambda=(1)$, we simply get back $V_1=H(n)$. However, not all irreducible representations of $\GGL_n$ are of this kind, for example the dual representation of $H(n)$. For notational convenience, let us write $H^\vee(n):=H(n)^\vee$ for the linear dual representation. 

To describe all of the irreducible representations of $\GGL_n$, we need to consider the mixed tensor powers of $H(n)$ and $H^\vee(n)$. For $p,q\ge 0$, we let
\begin{equation}
    K_{p,q}(n):=H(n)^{\otimes p}\otimes H^\vee(n)^{\otimes q}.
\end{equation}
Note that by applying the duality pairing $\lambda: H(n)\otimes H^\vee(n)\to \bQ$, we get a $\GGL_n$-equivariant map $\lambda_{i,j}:K_{p,q}(n)\to K_{p-1,q-1}(n)$ for all $p,q\ge 1$, $1\le i\le p$ and $1\le j\le q$. We define
\begin{equation}
    K_{p,q}^\circ(n):=\bigcap_{i,j} \ker(\lambda_{i,j}).
\end{equation}
\begin{definition}
    If $\lambda$ and $\mu$ are partitions, we call the pair $(\lambda,\mu)$ a \textit{bipartition} and if $|\lambda|=p$ and $|\mu|=q$, we say it is a bipartition of $(p,q)$ and write $|(\lambda,\mu)|=(p,q)$. We define the \textit{length} of a bipartition $(\lambda,\mu)$ as $l(\lambda,\mu)=l(\lambda)+\lambda(\mu)$.
\end{definition} For a bipartition $(\lambda,\mu)$, we define representations
\begin{equation}
    V_{\lambda,\mu}:=(S^\lambda\otimes S^\mu)\otimes_{\Sigma_p\times\Sigma_q} K_{p,q}^\circ(n).
\end{equation}
 Similarly to above, we have that $V_{\lambda,\mu}$ is non-zero as long as $l(\lambda,\mu)\le n$ and irreducible in this case. Furthermore, these are all the irreducible representations of $\GGL_n$, up to isomorphism.

The following important theorem tells us how $K^\circ_{p,q}(n)$ decomposes into irreducibles:

\begin{theorem}[Schur-Weyl duality]\label{thm:schurweyl}
    The $\Sigma_p\times\Sigma_q\times\GGL_n$-representation $K_{p,q}^\circ(n)$ decomposes into irreducibles as
    $$K_{p,q}^\circ(n)\cong\bigoplus_{|(\lambda,\mu)|=(p,q)} S^\lambda\otimes S^\mu\otimes V_{\lambda,\mu}.$$
\end{theorem}

\begin{proof}
    See for example \cite[Theorem 1.1]{Koike}.
\end{proof}

In order to use the categorical framework which we will describe in Section \ref{sec:walledBrauer} below, we will need to work with tensor powers of $H(n)$ and $H^\vee(n)$ over arbitrary finite sets. We therefore note that for any pair $S=(S_1,S_2)$ of finite sets, we can define $K_{S_1,S_2}(n)$ and $K_{S_1,S_2}^\circ(n)$ analogously to the definitions above. For brevity, we will often write $K_S(n)=K_{S_1,S_2}(n)$ and $K_S(n)^\circ=K_{S_1,S_2}^\circ(n)$

For a pair of finite sets $S=(S_1,S_2)$, we can also describe a decomposition of $K_{S}(n)$ in terms of the representations $K^\circ_{I}(n)$, for $I\subseteq S$. In order to do this, note that we have a form $\omega:\bQ\to H(n)\otimes H^\vee(n)$, which is dual to the pairing $\lambda$. We can describe it by sending 1 to the element corresponding to $\id_{H(n)}$ under the isomorphism $H(n)\otimes H^\vee(n)\cong\Hom_\bQ(H(n),H(n))$. If $e_i:=[x_i]\in H(n)$ are the standard basis elements, and we write $ e^\#_i$ for the dual basis elements in $H^\vee(n)$, we thus have
$$\omega(1)=\sum_{i=1}^ne_i\otimes  e^\#_i.$$
We remark that this means that $\lambda(\omega(1))=n$.

For any $s_1\in S_1$ and $s_2\in S_2$ we can apply the form to get a $\GGL_n$-equivariant map
$$\omega_{s_1,s_2}:K_{S\setminus(\{s_1\},\{s_2\})}(n)\to K_{S}(n).$$
We call this map an \textit{insertion}. The representation $K_{S}(n)$ satisfies the following:
\begin{proposition}\label{prop:decomp1}
    For any pair $S=(S_1,S_2)$ of finite sets we have
    \begin{equation}\label{eq:decomp1}
        K_{S}(n)= K^\circ_{S}(n)\oplus\sum_{\substack{s_1\in S_1\\ s_2\in S_2}}\omega_{s_1,s_2}(K_{S\setminus(\{s_1\},\{s_2\})}(n))
    \end{equation}
\end{proposition}

\begin{proof}
    We prove this similarly to \cite[Equation 17.12]{FultonHarris}. Letting $e_i:=[x_i]\in H(n)$ denote the standard basis, we have a standard inner product on $H(n)$, which we will denote by $\langle\cdot,\cdot\rangle$, given by $\langle e_i,e_j\rangle=\delta_{ij}$. This induces a dual metric on $H^\vee(n)$ and these two metrics together induce a metric on $K_S(n)=H(n)^{\otimes S_1}\otimes H^\vee(n)^{\otimes S_2}$, for any pair $S=(S_1,S_2)$ of finite sets , which we will also denote by $\langle\cdot,\cdot\rangle$, for simplicty. 

    Now let $a\in H(n)$ and $\beta\in H^\vee(n)$. Writing these as $a=\sum_{i=1}^n a_ie_i$ and $\beta=\sum_{i=1}^n b_i e^\#_i$, for $a_i,b_i\in\bQ$, we get
    \begin{align*}
        \langle\omega(1),a\otimes \beta\rangle=\sum_{i,j,k}a_jb_k\langle e_i,e_j\rangle\langle  e^\#_i, e^\#_k\rangle=\sum_{i,j,k}a_jb_k\delta_{ij}\delta_{ik}=\sum_i a_ib_i=\lambda(a,\beta).
    \end{align*}
    It follows from this that for any $s_1\in S_1$, $s_2\in S_2$, $\ker(\lambda_{s_1,s_2})=\mathrm{im}(\omega_{s_1,s_2})^\perp$, which implies (\ref{eq:decomp1}).        
\end{proof}

We may now apply this proposition inductively to $K_{S\setminus(\{s_1\},\{s_2\})}$ and decompose further, which corresponds to applying several insertions at once. To make this more precise, let us first make the following definition:
\begin{definition}
    If $S=(S_1,S_2)$ is a pair of finite sets, a \textit{bipartite matching} of $S$ is a partition of $S_1\sqcup S_2$ into disjoint ordered pairs $(a,b)$ such that $a\in S_1$ and $b\in S_2$. Let $\rM(S)$ denote the set of bipartite matchings of $S$.
\end{definition}

\begin{remark}
    Note that the a bipartite matching of $S=(S_1,S_2)$ is equivalent to a bijection between the two sets. However, we use this terminology in analogy with the one used in \cite{KupersRandal-Williams}. 
\end{remark}

If $I\subseteq S$ with $|I_1|=|I_2|$ and $m\in \rM(I)$, we thus get a $\GGL_n$-equivariant map
$$\omega_m:K_{S\setminus I}(n)\to K_{S}(n)$$
and we can decompose $K_{S}(n)$ as follows:

\begin{proposition}\label{prop:decomp2}
    For $S$ a pair of finite sets, the representation $K_{S}(n)$ decomposes as
    $$K_{S}(n)\cong \bigoplus_{i=0}^{\min(|S_1|,|S_2|)}\sum_{\substack{I\subseteq S\\ |I_1|=|I_2|=i}}\sum_{\substack{m\in\rM(I)}}\omega_m(K_{S\setminus I}^\circ(n)).$$
    For $n\ge |S|$, this is a direct sum
    $$K_S(n)\cong\bigoplus_{I\subseteq S}\bigoplus_{m\in \rM(I)}\omega_m(K_{S\setminus I}^\circ(n)).$$
\end{proposition}

\begin{proof}
    Let $s=\min(|S_1|,|S_2|)$. We start introducing the notation that for any $0\le i\le s$,
    $$K^{(i)}_{S}(n):=\sum_{\substack{I\subseteq S\\ |I_1|=|I_2|=i}}\sum_{\substack{m\in\rM(I)}}\omega_m(K_{S\setminus I}^\circ(n))\subseteq K_{S}(n).$$
     Thus we first want to show that
    $$K_{S}(n)=\bigoplus_{i=0}^{s} K_{S}^{{(i)}}(n).$$
    We prove this by induction on $s$. If $s=0$, we have $K_{S}^{(s)}(n)=K_{S}^\circ(n)=K_{S}(n)$, so the statement is trivially true. Now suppose that statement holds for some $s\ge 0$ and suppose that $\min(|S_1|,|S_2|)=s+1$. By first applying Proposition \ref{prop:decomp1} and then the inductive assumption, we get
    \begin{align*}
        K_{S}(n)&=K_{S}^{(0)}(n)\oplus\sum_{\substack{s_1\in S_1\\ s_2\in S_2}}\omega_{s_1,s_2}(K_{S\setminus(\{s_1\},\{s_2\})}(n))\\
        &=K_{S}^{(0)}(n) \oplus\sum_{s_1,s_2}\omega_{s_1,s_2}\left(\bigoplus_{i=0}^s K_{S\setminus(\{s_1\},\{s_2\})}^{(i)}(n)\right)\\
        &=\bigoplus_{i=0}^{s+1} K_{S}^{(i)}(n).
    \end{align*}
    Finally, let us assume that $n\ge |S|$. In this case, it follows by \cite[Theorem 2.7, 2.11]{BenkartBrauer} that the sum
    $$K^{(i)}_{S}(n):=\sum_{\substack{I\subseteq S\\ |I_1|=|I_2|=i}}\sum_{\substack{m\in\rM(I)}}\omega_m(K_{S\setminus I}^\circ(n))\subseteq K_{S}(n)$$
    is in fact an internal direct sum. Thus the second part of the proposition follows.
\end{proof}

In anticipation of the next section, we also note the following basic observation about insertions:

\begin{lemma}\label{lemma:contractioninsertion}
    If $v\in H(n)$, then $\lambda_{1,1}(v\otimes\omega(1))=v$. Similarly, if $\alpha\in H^\vee(n)$, then $\lambda_{1,2}(\omega(1)\otimes\alpha)=\alpha$.
\end{lemma}

\begin{proof}
    Via the isomorphism $H(n)\otimes H^\vee(n)\cong \Hom_\bQ(H(n),H(n))$ we have $K_{2,1}(n)\cong H(n)\otimes \Hom_{\bQ}(H(n),H(n))$ and under this identification, the map $\lambda_{1,1}:K_{2,1}(n)\to H(n)$ corresponds to the evaluation map. Since the isomorphism sends $\omega(1)$ to $\id_{H(n)}$ by definition, the first statement follows. The second statement is obtained in the corresponding way, using the isomorphism $H(n)\otimes H^\vee(n)\cong\Hom_{\bQ}(H^\vee(n),H^\vee(n))$.
\end{proof}

\subsubsection{The arithmetic subgroup $\GL_n(\bZ)$} If $\mathbf{G}$ is an algebraic group defined over $\bQ$, recall that an \textit{arithmetic subgroup} is a subgroup $G\subset\mathbf{G}(\bQ)$ which is commensurable to $\mathbf{G}(\bZ)$.  Our main example will be $\GL_n(\bZ)$ which is an arithmetic subgroup of $\GGL_n$. Note that the standard representation $\Aut(F_n)\to\GL(H(n))$ factors through $\GL_n(\bZ)$. 

\begin{remark}
    Sometimes it will also be useful to consider the algebraic subgroup $\mathbf{SL}_n^{\pm}\subset\GGL_n$, defined by the polynomial function $\det(X)^2-1$. Note that $\GL_n(\bZ)$ is an arithmetic subgroup of $\mathbf{SL}_n^\pm$ as well. Furthermore, since $\SL_n(\bZ)$ is Zariski dense in $\SL_n(\bQ)$, it follows that $\GL_n(\bZ)$ is Zariski dense in $\SL_n^\pm(\bQ)$. In contrast $\GL_n(\bZ)$ is \textit{not} Zariski dense in $\GL_n(\bQ)$ (consider again the polynomial function $\det(X)^2-1$).
\end{remark}

\subsubsection{Invariant theory of $\GGL_n$ and $\GL_n(\bZ)$} As above, we write $\lambda:H(n)\otimes H(n)^*\to\bQ$ and $\omega:\bQ\to H(n)\otimes H(n)^*$ for the natural pairing and form, respectively. Since these maps are $\GGL_n(\bQ)$-equivariant, the image of $\omega$ is in particular a $\GGL_n(\bZ)$-invariant tensor and we can use the form to construct tensors of this kind, by applying several insertions, as in the previous subsection. In particular, we get a linear map:
\begin{equation}\label{eq:invariants}
    \bQ\{\text{Bipartite matchings of }S\}\to K_S(n)^{\GL_n(\bZ)},
\end{equation}
defined by $m\mapsto \omega_m(1)$. The map satisfies the following proposition:

\begin{proposition}\label{prop:invariantsofK}
    For any $p,q\ge 0$, the map (\ref{eq:invariants}) is a surjection. Furthermore, it is an isomorphism as long as $n\ge |S|+3$.
\end{proposition}

\begin{proof}
    For surjectivity, we use the first fundamental theorem of invariant theory for $\mathbf{SL}_n(\bQ)$ (see for example \cite[Chapter 11.1.2, Theorem 3]{ProcesiLieGroups}\footnote{Here, the theorem is stated for the field $\bC$, but by the reasoning at the beginning of the same chapter, the results are also valid over $\bQ$.}), which can be stated as saying that the the $\mathbf{SL}_n(\bQ)$-invariants are generated by elements of the form $\omega_m(1)$, and possibly by determinant representations in $H(n)^{\otimes S_1}$ or $H^\vee(n)^{\otimes S_2}$. However, since the determinant representation of $\mathbf{SL}_n^\pm(\bQ)$ is the sign representation, it follows that the elements $\omega_m(1)$ generate the $\mathbf{SL}_n^\pm(\bQ)$-invariants. Since $\GL_n(\bZ)$ is Zariski dense in $\mathbf{SL}_n^\pm(\bQ)$, it follows that the map (\ref{eq:invariants}) is indeed surjective.
    
    To prove the second statement, it is possible to use the second fundamental theorem of invariant theory, but we can also simply use that the homomorphism $\Aut(F_n)\to\GL_n(\bZ)$ is surjective, so the proposition follows from \cite[Theorem A]{Lindell2022stable}, which is the same as Theorem \ref{thm:twistedcoeffs} below.
\end{proof}

\subsection{Borel vanishing.} As we have already noted, the first main input we will use to compute stable cohomology of $\IA_n$ is the stable cohomology of $\Aut(F_n)$ with the twisted coefficients $K_S(n)$. The second main input is the stable cohomology of $\GL_n(\bZ)$ with twisted coefficients, which is known for a large class of coefficients due to work of Borel. 

\begin{definition}
    If $\mathbf{G}$ is an algebraic group, $G$ is an arithmetic subgroup and $V$ is a finite dimensional $\bQ$-representation of $G$, then $V$ is \textit{algebraic} if there is a morphism of algebraic groups $\mathbf{G}\to\GGL(V)$ such that taking $\bQ$-points and restricting to $G$ yields the representation $G\to\GGL(V)$. We say that a representation $V$ of $G$ is \textit{almost algebraic} if there is a finite index subgroup $G'\subseteq G$ such that the restriction $G'\to\GGL(V)$ is algebraic.
\end{definition}

\begin{remark}\label{rmk:almostalgebraic}
    If $\rho:\mathbf{SL}_n^\pm\to\GGL(V)$ is an irreducible algebraic representation, it lifts to an algebraic representation of $\GGL_n$ in a unique way (see \cite[§15.5]{FultonHarris}). For $V$ a finite dimensional of $\GL_n(\bZ)$, it thus follows that the notion of $V$ being almost algebraic does not depend on whether we consider $\GL_n(\bZ)$ as an arithmetic subgroup of $\GGL_n$ of of $\mathbf{SL}_n^{\pm}$.
\end{remark}
\

\noindent The main reason to consider almost algebraic representations is the following theorem:

\begin{theorem}[{\cite[Page 109]{SerreArithmeticGroups}}]
    If $\mathbf{G}$ is a simple algebraic group of $\bQ$-rank $\ge 2$ defined over $\bQ$, $G$ is an arithmetic subgroup of $\mathbf{G}$ and $V$ is a finite-dimensional representation of $G$, then $V$ is almost algebraic.
\end{theorem}

\noindent As $\mathbf{SL}_n^\pm$ contains $\mathbf{SL}_n$ as a finite index subgroup and $\mathbf{SL}_n$ is a simple algebraic group, the theorem applies to almost algebraic representations of arithmetic subgroups of $\mathbf{SL}_n^\pm$. In particular finite-dimensional representations of $\GL_n(\bZ)$ are almost algebraic. Considering almost algebraic representations will be useful in the proof of the following theorem, which is a slightly stronger version of the Borel vanishing theorem for $\GGL_n$ (cf. \cite[Theorem 2.3]{KupersRandal-Williams}):

\begin{theorem}[Borel vanishing]\label{thm:BorelVanishing}
    Let $G$ be an arithmetic subgroup of $\GGL_n$ and $V$ be a finite dimensional representation of $G$. For $n\ge *+2$, the maps
    $$H^*(\GGL_\infty,\bQ)\otimes V^{G}\to H^*(G,\bQ)\otimes V^{G}\to H^*(G,V)$$
    are isomorphisms, where $H^*(\GGL_\infty,\bQ)=\exterior{*}\{x_{5},x_9,\ldots\}$.
\end{theorem}

\begin{proof}
    First, for $V$ a finite dimensional and algebraic $G$-representation, this is due to Borel \cite{Borel1}\cite{Borel2}, but with a different range. The improved range is due to Li and Sun \cite[Example 1.10]{LiSun}. For $V$ finite dimensional and thus almost algebraic, we may apply the same transfer argument as in \cite{KupersRandal-Williams}, but let us spell it out in detail for completeness. Let $G'\le G$ be a finite index subgroup of $G$ such that the restriction of $G\to\GGL(V)$ is algebraic. We may assume that $G'$ is a normal subgroup of $G$ by taking the normal core. We then have a commutative diagram
    \[\begin{tikzcd}
        H^*(G,\bQ)\otimes V^{G}\arrow[rr]\arrow[d]&&H^*(G,V)\arrow[d]\\
        H^*(G',\bQ)^{G/G'}\otimes (V^{G'})^{G/G'}\arrow[r,hook]& (H^*(G',\bQ)\otimes V^{G'})^{G/G'}\arrow[r] & H^*(G',V)^{G/G'},
    \end{tikzcd}\]
    where the top horizontal arrow is given by cup product, the first bottom horizontal arrow is the natural inclusion, the second bottom horizontal arrow is induced by the cup product and the vertical arrows are given by transer maps. By \cite[Proposition 10.4]{BrownGroupCohomology}, the vertical arrows are isomorphisms and by the first part of the proof, the second bottom horizontal arrow is as well, in a range of degrees. Furthermore, by the first part of the theorem $H^*(G,\bQ)\cong H^*(G',\bQ)\cong H^*(\GGL_\infty,\bQ)$, so it follows that $G$ acts trivially on $H^*(G',\bQ)$ in the given range of degrees, and thus the first lower horizontal arrows is also an isomorphism in this range. It follows that the top horizontal arrows is also an isomorphism in the given range of degrees.
\end{proof}

As in \cite[Section 2.1.2]{KupersRandal-Williams}, let us point out that in particular, if $V$ is finite dimensional, the theorem implies that for $n\ge 3$, we have $H^1(\GL_n(\bZ),V)=0$, so by the long exact sequence in group cohomology we get that $[-]^{\GL_n(\bZ)}$ is an exact functor on the category of finite dimensional representations. Furthermore, we have that if $V$ and $W$ are finite dimensional then
$$\mathrm{Ext}^1_{\GL_n(\bZ)}(V,W)\cong H^1(\GL_n(\bZ),V^\vee\otimes W)=0,$$
so all extensions of finite dimensional representations split.

\section{The walled Brauer category}\label{sec:walledBrauer} 
Our goal is to study $H^*(\IA_n,\bQ)$ simultaneously as a ring and as a $\GL_n(\bZ)$-representation, which we will do using the categorical approach of \cite{KupersRandal-Williams}. In this section we recall the necessary categorical background from that paper and apply this to the \textit{walled Brauer category}. 

\subsection{Categorical background} We start by briefly summarizing the setup and results of \cite[Section 2.2]{KupersRandal-Williams}. This can be safely skipped by the reader familiar with that paper.

We start from two categories. First, we let $(\cA,\otimes,1_\cA)$ denote a symmetric monoidal $\bQ$-linear abelian category which we assume has all finite enriched colimits. In our applications, this will generally be the category $\Rep(\GL_n(\bZ))$.

Secondly, we let $(\Lambda,\oplus,0_\Lambda)$ be a symmetric monoidal $\bQ$-linear category, where we have a partial order $\le$ on the isomorphism classes of objects, defined by saying that $[x]\le [y]$ whenever $\Hom_\Lambda(x,y)\neq 0$. Furthermore, we will assume that any object $x\in\Lambda$ only has nonzero morphisms to a finite number of objects, up to isomorphism. In our applications, $\Lambda$ will be the \textit{downward walled Brauer category}, which we introduce in the next subsection. 

We will also consider the abelian categories $\cA^\Lambda$ and $(\bQ\dashmod)^\Lambda$, of functors from $\Lambda$ to the respective categories. 

There are two subclasses of objects of these categories that will be of particular interest, but which we can define in any abelian (resp. symmetric monoidal) category:

\begin{definition}
    We say that an object $B$ of an abelian category $\mathcal{B}$ is of \textit{finite length} if there is a finite sequence
$$0\hookrightarrow B_1\hookrightarrow B_2\hookrightarrow\cdots\hookrightarrow B_k=B$$
of monomorphisms, such that at each step, the only quotients of the cokernel $\coker(B_{i+1}\hookrightarrow B_{i})$ are 0 and the cokernel itself. Note that in $\bQ\dashmod$ this just corresponds to being a finite dimensional vector space. We write $\mathcal{B}^f$ for the full subcategory of objects of finite length.
\end{definition}

\begin{remark}\label{rmk:tensoredover}
    The category $\cA$ is tensored over $(\bQ\dashmod)^f$, which means that for $V\in(\bQ\dashmod)^f$ and $A\in\cA$ there is an object $A \tensoredover V\in \cA$. The object $A \tensoredover V$ is characterized by a natural isomorphism $\Hom_\cA(A \tensoredover V,-)\cong \Hom_\bQ(V,\Hom_\cA(A,-))$. In particular, there is a functor 
    $$1_\cA\tensoredover-:(\bQ\dashmod)^f\to\cA^f,$$
    which has a right adjoint $\Hom_\cA(1_\cA,-):\cA^f\to(\bQ\dashmod)^f$. Note that for $\cA=\Rep(G)$, where $G$ is a group, the functor $1_\cA\tensoredover-$ is simply given by considering a vector space as a trivial representation, and the right adjoint is given by taking the $G$-invariants of a $G$-representation.
\end{remark}

\begin{definition}
    We say that an object $C$ of a symmetric monoidal category $\mathcal{C}$ is \textit{dualizable} if there is another object $X^\vee$, which we call a \textit{dual} of $X$, together with morphisms $\epsilon:X\otimes X^\vee\to 1$ and $\eta:1\to X\otimes X^\vee$ satisfying the triangle identities. We write $\mathcal{C}^d$ for the full subcategory of dualizable objects. A dual is necessarily unique up to isomorphism.
\end{definition}

If $K\in(\cA^d)^{\Lambda}$ (in \cite{KupersRandal-Williams}, a functor that has dualizable values is called a \textit{kernel}), we can define a functor $K^\vee:\Lambda^{\mathrm{op}}\to\cA^d$, which for $x\in\Lambda$ given by $K^\vee(x)=K(x)^\vee$ and for $(f:x\to y)\in\Hom_\Lambda(x,y)$ is given by the composition
\[\begin{tikzcd}
    K(y)^\vee\arrow[r,"\id\otimes\eta"]&K(y)^\vee\otimes K(x)\otimes K(x)^\vee\arrow[rr,"\id\otimes K(f)\otimes\id"]&&K(y)^\vee\otimes K(y)\otimes K(x)^\vee\arrow[r,"\epsilon\otimes\id"]&K(x)^\vee.
\end{tikzcd}\]
If $M\in(\cA^\Lambda)^f$, this means that we can define a coend 
\begin{equation}\label{eq:coend}
    K^\vee\otimes^\Lambda M:=\int^{x\in\Lambda}K^\vee(x)\otimes M(x)\in\cA,
\end{equation}
which defines a functor $K^\vee\otimes^\Lambda-:(\cA^\Lambda)^f\to\cA$. In \cite{KupersRandal-Williams}, the existence of this coend is proven using that is may also be expressed as the coequalizer of the diagram
\begin{equation}
    \bigoplus_{x,y\in\Lambda}(K^\vee(y)\otimes M(x)) \tensoredover \Hom_\Lambda(x,y)\rightrightarrows\bigoplus_{x\in\Lambda} K^\vee(x)\otimes M(x),
\end{equation}
which we will use a lot in our applications. There, the functor $K$ will be defined using the representations $K_{S}(n)$ from the previous section and $M$ will be a functor defined using the stable cohomology of $\Aut(F_n)$ with bivariant twisted coefficients. The stable cohomology of $\IA_n$ will be expressed in terms of the coend (\ref{eq:coend}). Since the cohomology is a ring, we will next recall the multiplicativity conditions necessary to endow the coend with such structure.

The categories $(\cA^\Lambda)^f$ and $(\bQ\dashmod^\Lambda)^f$ are equipped with symmetric monoidal structures given by Day convolution, i.e.\ for functors $M,N$ we can define
$$(M\otimes_\Lambda N)(x):=\bigoplus_{a\oplus b=x}M(a)\otimes N(a).$$
In our applications, the functor $K$ will be strong symmetric monoidal and the functor $M$, from the previous paragraph, will be lax symmetric monoidal, i.e.\ a commutative ring object in $\cA^\Lambda$. Thus the ring structure on the coend will follow from the following proposition:
\begin{proposition}[{\cite[Proposition 2.11]{KupersRandal-Williams}}]
    If $K\in(\cA^d)^\Lambda$ is strong symmetric monoidal, so is the functor $K^\vee\otimes^\Lambda-:(\cA^\Lambda)^f\to\cA$. 
\end{proposition}

The last thing we need to recall is a key lemma used to detect isomorphisms in $\cA$. For this we need some more notation. We start by fixing $K\in(\cA^d)^\Lambda$. Now we let $\Delta:\cA\to(\bQ\dashmod)^\Lambda$ be the composite functor
\[\begin{tikzcd}\cA\arrow[rr,"K\otimes -"]&&\cA^\Lambda\arrow[rrr,"\Hom_{\cA}{(1_\cA,-)}"]&&&(\bQ\dashmod)^\Lambda.
\end{tikzcd}\]
Furthermore, we let $\cA_K$ denote the subcategory whose objects are summands of objects of the form $K(x)^\vee$, for any $x\in\Lambda$. Finally, we let $\cA_K^\circ$ be the category of objects $X$ such that $\Hom_\cA(1_\cA,X\otimes K(x))=0$ for any $x\in\Lambda$.

\begin{lemma}[{\cite[Lemma 2.12]{KupersRandal-Williams}}]\label{lemma:testisos}
    Let $f:A\to B$ be a morphism in $\cA$.\begin{enumerate}
        \item If $\Delta(f)$ is injective, then $\ker(f)\in\cA_K^\circ$.
        \item If $A\in\cA_K$, $\mathrm{Ext}_{\cA}^1(K(x)^\vee,K(y)^\vee)=0$ for all $x,y\in\Lambda$ and $\Delta(f)$ is bijective, then $\coker(f)\in\cA_K^\circ$.
    \end{enumerate}
\end{lemma}

\subsection{The walled Brauer category}

Let $G=\GL_n(\bZ)$ and $\cA:=\Rep(G)$ be the category of representations of $G$. We shall assume that $n\ge 3$, so that $[-]^G$ is an exact functor and and all extensions split.

\begin{definition}
    The \textit{walled Brauer category} $(\wBr_d,\oplus,0)$ of charge $d\in\bQ$ is the symmetric monoidal $\bQ$-linear category defined as follows:\begin{itemize}
        \item The objects of $\wBr_d$ are pairs $S=(S_1,S_2)$ of finite sets.
        \item The space of morphisms $\wBr_d(S,T)$ is the $\bQ$-vector space with basis given by quadruples $(f,g,m_{S},m_{T})$, where $f$ is a bijection from a subset $S'_1\subset S$ to a subset $T'_1\subset T_1$, $g$ is a bijection from a subset $S'_2\subset S_2$ to a subset $T'_2\subset T_2$, $m_{S}$ is a bipartite matching of $(S_1\setminus S'_1, S_2\setminus S'_2)$
        and $M_{T}$ is a bipartite matching of $(T_1\setminus T'_1, T_2\setminus T'_2)$. Note that this means that a morphism in $\wBr_d(S,T)$ corresponds precisely to a bipartite matching of $(S_1\sqcup T_2, T_1\sqcup S_2)$.
        The data of a morphism can be recorded graphically as illustrated in Figure \ref{fig:wBr}. 
\begin{figure}[h]
    \centering
    \includegraphics[scale=0.55]{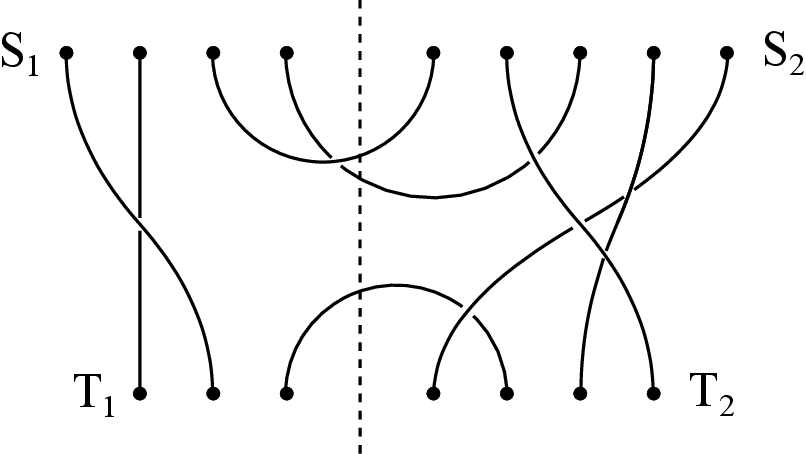}
    \caption{Illustratation of a morphism in the walled Brauer category. Note that the order of crossings is not part of the data and does not matter for the morphism, but is only included to make the picture easier to read.}
    \label{fig:wBr}
\end{figure}
\item The composition $\wBr_d(S,T)\otimes\wBr_d(T,U)\to\wBr_d(S,U)$ is by defined concatenating pictures like that in Figure \ref{fig:wBr} vertically, removing closed components and multiplying by $d$, raised to the power of the number of closed components removed.
\begin{figure}[h]
    \centering
    \includegraphics[scale=0.48]{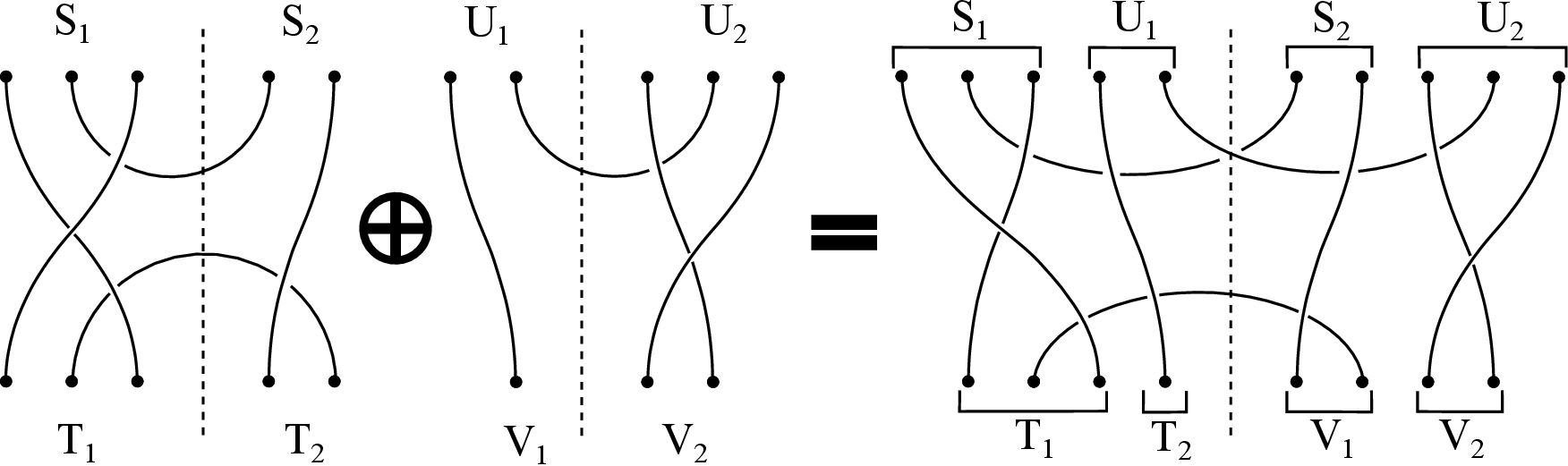}
    \caption{The monoidal product of two morphisms in $\wBr_n$.}
    \label{fig:wBrM}
\end{figure}
\item The monoidal product is given on objects by $(S_1,S_2)\oplus(T_1,T_2):=(S_1\sqcup T_1,S_2\sqcup T_2)$ and on morphisms by concatenating diagrams like that in Figure \ref{fig:wBr} horizontally, with the obvious symmetry. The monoidal unit $0$ is the pair $(\varnothing,\varnothing)$. 
    \end{itemize}  
\end{definition}

As in \cite{KupersRandal-Williams}, will be particularly interested in the following subcategory of $\wBr_d$:

\begin{definition}
    The \textit{downward walled Brauer category} $\dwBr\subset \wBr_d$ is the subcategory with the same objects as $\wBr_d$, but only with 
    $$\Hom_{\dwBr}(S,T)=\bQ\{(f,g,m_S,\varnothing)\in\Hom_{\wBr_n}(S,T)\}.$$
     Note that this means that no closed components can appear when composing morhpisms, so this subcategory is independent of the charge $d$. We will write $i:\dwBr\to\wBr_d$ for the inclusion functor.
\end{definition}

\subsubsection{Left Kan extensions along the inclusion functor}\label{subsec:leftKanextension} For $\cC$ a symmetric monoidal $\bQ$-linear abelian category (either $\cA$ or $\bQ\dashmod$), we will several times consider the functor $i_*:\cC^{\dwBr}\to\cC^{\wBr_n}$ given by left Kan extension along the inclusion functor $i$, so let us described this functor in some detail. Let $A:\dwBr\to\cC$ be a functor. The functor $i_*A:\wBr_n\to\cC$ may then be expressed using the coend
\begin{equation}
    i_*A:=\int^{T\in\dwBr}\Hom_{\wBr_n}(T,-)\tensoredover A(T)\in\cC^{\wBr_n},
\end{equation}
which may also be expressed as the coequalizer of the diagram
$$\bigoplus_{(U\to T)\in\dwBr}\Hom_{\wBr_n}(T,-)\tensoredover A(U)\rightrightarrows\bigoplus_{T\in\dwBr}\Hom_{\wBr_n}(T,-)\tensoredover A(T).$$
We note that any morphism $(f,g,m_T,m_S)\in\Hom_{\wBr_n}(T,S)$ can be written as a composition
$$(\id_{S'_1},\id_{S'_2}, \varnothing, m_S)\circ(f,g,m_T,\varnothing),$$
and since the first map in the composition is in $\dwBr(T,S')$, it follows that for $S\in\wBr_n$, we have
\begin{equation}\label{eq:KanFormula}
    i_*(A)(S)\cong\bigoplus_{I\subseteq S}\Hom_{\wBr_n}(\varnothing,S\setminus I)\tensoredover A(I)
\end{equation}
Note that $\Hom_{\wBr_n}(\varnothing, S\setminus I)=\bQ M(I)$, so for brevity, we denote a basis element of $\Hom_{\wBr_n}(\varnothing, S\setminus I)$ by $m_{S\setminus I}$ instead of $(\varnothing,\varnothing,\varnothing,m_{S\setminus I})$. To describe the functor $i_*(A)$ on morphisms, we consider two main cases:\begin{enumerate}[(1)]
    \item First, consider a morphism of the form $F=(f,g,\varnothing,m_T):S\to T$, i.e.\ where $f:S_1\to T_1$ and $g:S_2\to T_2$ are injections and $m_T$ is a bipartite matching of $(T_1\setminus f(S_1),T_2\setminus g(S_2))$. For $I\subseteq S$, let $J:=(f(I_1),g(I_2))\subseteq T$. Since $m_T$ is a (possibly empty) bipartite matching of $T\setminus J$, the morphism $F$ restricts to morphisms $F_I=(f|_{I_1},g|_{I_2},\varnothing,\varnothing)\in\Hom_{\wBr_n}(I,J)$ and $F_{S\setminus I}=(f|_{S_1\setminus I_1},g|_{S_2\setminus I_2},\varnothing,m_T)\in\Hom_{\wBr_n}(S\setminus I,T\setminus J)$. Note that $F_I$ is a morphism in $\Hom_{\dwBr}(I,J)$ and we have
    $$i_*(A)(f,g,\varnothing,m_T)=\bigoplus_{I\subseteq S}\Hom_{\wBr_n}(\varnothing,F_{S\setminus I})\otimes A(F_I).$$
    \item Secondly, for $S'=(\{s_1\},\{s_2\})\subseteq S$, let $S'':=S\setminus S'$ and let us consider a morphism of the form 
    $$F=(\id_{S_1''},\id_{S_2''},(s_1,s_2),\varnothing)\in\Hom_{\wBr_n}(S,S'').$$
    Now consider $I\subseteq S$. Here we must consider four different cases:\begin{enumerate}[(a)]
        \item First, suppose $S'\subseteq I$. In this case, the morphism $F$ restricts to the identity on $S\setminus I$ and on $I$ to the morphism $F':=(\id_{I_1\setminus\{s_1\}},\id_{I_2\setminus\{s_2\}},(s_1,s_2),\varnothing)$, which is a morphism in $\Hom_{\dwBr}(I,I\setminus S')$, so on the summand of $i_*(A)(S)$ corresponding to $I$, $i_*(A)(F)$ is given by 
        $$\Hom_{\wBr_n}(\varnothing,\id_{S\setminus I})\otimes A(F').$$
        \item Next, suppose $S'\subset S\setminus I$. Similarly as in the previous case, the morphism $F$ restricts to the identity on $I$ an on $S\setminus I$ to the morphism $F':=(\id_{S_1\setminus\{s_1\}},\id_{S_2\setminus\{s_2\}},(s_1,s_2),\varnothing)$. On the summand of $i_*(A)(S)$ corresponding to $I$, $i_*(A)(F)$ is thus given by
        $$\Hom_{\wBr_n}(\varnothing,F')\otimes A(\id_I).$$
        In this case, we can note that since $S'\subset S\setminus I$, the morphism $\Hom_{\wBr_n}(\varnothing,F')$ is explicitly given by removing $(s_1,s_2)$ from the matching of the morphism in $\Hom_{\wBr_n}(\varnothing, S\setminus I)$ and multiplying by $n$.
        \item Next, suppose $s_1\in I_1$ and $s_2\in S_2\setminus I_2$. This case is a bit more complicated and we cannot use the functoriality of the tensor factors of each summand directly. Instead, we describe the morphism for each basis element of $\Hom_{\wBr_n}(\varnothing, S\setminus I)$ separately. For $m_{S\setminus I}\in\Hom_{\wBr_n}(\varnothing,S\setminus I)$, let $s_1'\in S_1\setminus I_1$ be such that $(s_1',s_2)\in m_{S\setminus I}$. We define $I'=(I_1',I_2')\subseteq S$ by $I_1':=(I_1\setminus\{s_1\})\sqcup\{s_1'\}$ and $I_2':=I_2$. We let $F'=(f',\id_{I_2},\varnothing,\varnothing):I\to I'$ be the morphism where
        \begin{align*}
            f'(s)=\begin{cases}s_1'&\text{ if }s=s_1,\\ s&\text{ otherwise}.\end{cases}
        \end{align*}
        Similarly, we have a linear map $T:\bQ\{m_{S\setminus I}\}\to\bQ\{m_{S''\setminus I'}\}$ given by $m_{S\setminus I}\mapsto m_{S\setminus I}\setminus(s_1',s_2)$. On the summand $\bQ\{m_{S\setminus I}\}\tensoredover A(I)$ the morphism $i_*(A)(F)$ is then given by
        $$i_*(A)(F)=T\otimes A(F')$$
        
        \item Finally, the case $s_1\in S_1\setminus I_1$ and $s_2\in I_2$ is determined in the same way as the previous case.
    \end{enumerate}
\end{enumerate}
Any morphism in $\wBr_n$ is a composition of a number of morphisms of the forms (1) and (2), so this completely describes the functor.

\begin{example}
    Now let us apply this description to an explicit example, which will be of particular interest. For $n\ge 2$, we define a functor $K:\wBr_n\to\Rep(G)$, given on objects by $K(S)=K_{S}(n)$ and for $(f,g,m_S,m_T):S\to T$,  $K(f,g,m_S,m_T)$ is the composition
\begin{align*}
   K_{S}(n)\xrightarrow{\lambda_{m_S}}K_{S'}(n)\xrightarrow{f\otimes g}  K_{T'}(n) \xrightarrow{\omega_{m_T}}K_{T}(n),
\end{align*}
where we by abuse of notation also use $f$ and $g$ to denote the maps given by permuting the tensor factors according to these bijections. 
\end{example} 

We also define a functor $K^\circ:\dwBr\to\Rep(G)$ which on objects is given by $K(S)=K_{S}^\circ(n)$ and on morphisms is defined just as $K$. We then have the following proposition, which is essentially a direct consequence of Proposition \ref{prop:decomp2}:

\begin{proposition}\label{prop:KanextensionK}
    There is a natural transformation $\psi:i_*K^\circ\to K$,  of functors $\wBr_n\to\Rep(G)$, where $\psi_S$ is surjective for all $S\in\wBr_n$ and an isomorphism if $n\ge |S|$.
\end{proposition}

\begin{proof}
    First, we note that for $S\in\wBr_n$ and any $I\subseteq S$, we can define a $G$-equivariant map $\omega_{(-)}(1):\Hom(\varnothing,S\setminus I)\to K(S\setminus I)^G$ by $m\mapsto \omega_m(1)$. Using (\ref{eq:KanFormula}) for $i_*K^\circ$, we can thus for any $S\in\wBr_n$ define a $G$-equivariant map  
    $$\psi_S:\bigoplus_{I\subseteq S}\omega_{(-)}(1)\otimes\mathrm{inc}_I:i_*(K^\circ)(S)\to K(S),$$
    where $\mathrm{inc}_I:K^\circ(I)\to K(I)$ denotes the inclusion, for any $I\subseteq S$. Using the above description of the functor $i_*K^\circ$ on morphisms, it is a straightforward verification that the maps $\psi_S$ define a natural transformation. For the morphisms in (2b), one uses the fact that $\lambda(\omega(1))=n$ and for (2c) and (2d) one uses Lemma \ref{lemma:contractioninsertion}. It follows by Proposition \ref{prop:decomp2} that $\psi_S$ is surjective for any $S$ and an isomorphism as long as $n\ge |S|$. 
\end{proof}

By dualizing the values on objects and on morphisms, we also get a functor $K^\vee:\wBr_n^{\mathrm{op}}\to \Rep(G)$. Our main results will be proven using the following proposition, which is analogous to \cite[Proposition 2.16]{KupersRandal-Williams}. In what follows, $i^*:\cA^{\wBr_n}\to\cA^{\dwBr}$ denotes the restriction functor, i.e.\ the functor given by precomposing a functor with the inclusion functor $i$.

\begin{proposition}\label{prop:wBr1}
    Let $A\in (\bQ\text{-}\mathrm{mod})^{\dwBr}$ be a functor of finite length, and $B$ be a representation of $G$ (not necessarily finite dimensional). Given a natural isomorphism
    $$\phi^{\wBr_n}:i_*(A)\to[K\otimes B]^G,$$
    of functors in $(\bQ\dashmod)^{\wBr_n}$, there is an induced map
    $$\phi:i^*(K^\vee)\otimes^{\dwBr}(1_{\Rep(G)} \tensoredover A)\to B,$$
    of $G$-representations, which is an isomorphism onto the maximal algebraic subrepresentation of $B$. Furthermore, for any bipartition $(\lambda,\mu)$ of $(p,q)$ and for $n\gg p+q$, the multiplicity of the irreducible $G$-representation $V_{\lambda,\mu}$ in $B$ is the same as that of the irreducible $\Sigma_p\times\Sigma_q$-representation $S^\lambda\otimes S^\mu$ in $A(p,q)$.
\end{proposition}

\begin{proof}
    First, let us explain why we do not need to add any assumption on the dimension of the representation $B$, following \cite{ErratumKRW}. For any $S\in\wBr_n$, we have $[K(S)\otimes B]^G\cong\Hom_G(K(S)^\vee, B)$ and since $K(S)^\vee$ is algebraic, any map in $\Hom_G(K(S)^\vee, B)$ factors through the inclusion $B^{\alg}\subseteq B$. Thus the natural transformation 
    $\phi^{\wBr_n}:i_*(A)\to [K\otimes B]^G$ factors through the functor $[K\otimes B^{\alg}]^G$. Since the inclusion map $[K(S)\otimes B^\alg]^G\to [K(S)\otimes B]^G$ is injective, for any $S\in\wBr_n$ and $\phi^{\wBr_n}$ is assumed to be a natural isomorphism, the natural transformation $i_*(A)\to[K\otimes B^{\alg}]^G$, through which it factors, is thus necessarily an isomorphism as well. For the rest of the proof, we may thus replace $B$ by $B^{\alg}$ and we will by abuse of notation denote the map $i_*(A)\to[K\otimes B^{\alg}]^G$ by $\phi^{\wBr_n}$ as well.

    Next, we note that since $A$ is assumed to have finite length, the assumption that $\phi^{\wBr_n}$ is an isomorphism implies that $[K(S)\otimes B^{\alg}]^G\cong\Hom_G(K(S)^\vee,B^{\alg})$ is finite dimensional for every $S\in \wBr_n$. Furthermore, since $A$ is of finite length, we have that there exists some $N\ge 0$ such that $[K(S)\otimes B^{\alg}]^G=0$ for all $S$ with $|S|>N$. Since every irreducible algebraic $G$-representation $V_{\lambda,\mu}$ appears in $K(S)^\vee$ for some choice of $S$, it follows that $\Hom_G(V_{\lambda,\mu},B^{\alg})$ is also finite dimensional for every bipartition $(\lambda,\mu)$, and zero if $|\lambda|+|\mu|>N$. For each bipartition $(\lambda,\mu)$ we have an evaluation map $V_{\lambda,\mu}\otimes \Hom_G(V_{\lambda,\mu},B^{\alg})\to B^{\alg}$, and by adding these together we get a surjective map
    $$\bigoplus_{(\lambda,\mu)} V_{\lambda,\mu}\otimes \Hom_G(V_{\lambda,\mu},B^{\alg})\to B^{\alg}.$$
    Since $V_{\lambda,\mu}$ is non-zero only if $l(\lambda,\nu)\le n$, and $\Hom_G(V_{\lambda,\mu},B^{\alg})$ is only nonzero for $|\lambda|+|\mu|\le N$, it follows that there is only a finite number of non-zero summands in this direct sum, and so the domain of this map is finite dimensional as well. Thus $B^{\alg}$ is finite dimensional.

    Now, let us move on to prove the first statement of the proposition. This proof is essentially the same as that of \cite[Proposition 2.16]{KupersRandal-Williams}, but we include the details for completeness. We start by showing the existence of the map $\phi$. By the adjunction between the functors $i_*$ and $i^*$, the natural transformation $\phi^{\wBr_n}$ defines a natural transformation 
    $$\phi^{\dwBr}:A\to i^*[K\otimes B^{\alg}]^G=[i^*(K\otimes B^{\alg})]^G.$$
    However, the functor $[-]^G:\Rep(G)^f\to (\bQ\dashmod)^f$ is the same as $\Hom_{\Rep(G)}(1_{\Rep(G)},-)$, which by Remark \ref{rmk:tensoredover} is right adjoint to the functor 
    $$1_{\Rep(G)}\tensoredover -:(\bQ\dashmod)^f\to\Rep(G)^f,$$
    giving us a new natural transformation
    $$1_{\Rep(G)}\tensoredover A\to i^*(K\otimes B^{\alg})$$
    For any $S\in\dwBr$, we thus get a $G$-equivariant map
    $$(1_{\Rep(G)}\tensoredover A)(S)\to i^*K(S)\otimes B^{\alg}$$
    and thus by adjunction a $G$-equivariant map
    $$i^*K^\vee(S)\otimes(1_{\Rep(G)}\tensoredover A)(S)\to B^{\alg}.$$
    This induces the $G$-equivariant map
    $$\phi:i^*K^\vee\otimes^{\dwBr}(1_{\Rep(G)}\tensoredover A)\to B^{\alg}.$$
    To prove that this is an isomorphism, we will use Lemma \ref{lemma:testisos}. In this setting, we note that $\Delta:\Rep(G)\to (\bQ\dashmod)^{\dwBr}$ is the functor $[i^*K\otimes -]^G$, so we will start by proving that 
    $$\Delta(\phi):[i^*K\otimes(i^*(K^\vee)\otimes^{\dwBr}(1_{\Rep(G)}\tensoredover A))]^G\to [i^*K\otimes B^{\alg}]^G$$
    is a natural isomorphism, by using the assumption that $\phi^{\wBr_n}$ is. In order to do this, will show that for any $S\in\wBr_n$, the map $\phi^{\wBr_n}_S$ factors as a surjection followed by the map $\Delta(\phi)$. 
    
    We start by unraveling the definition of $\Delta(\phi)$. For any $S\in\dwBr$, we can write the source of this natural transformation evaluated at $S$ as
    $$\left[\int^{T\in\dwBr}K(S)\otimes K(T)^\vee\tensoredover A(T)\right]^G$$
    and since $[-]^G$ is exact, this is the same as
    $$\int^{T\in\dwBr}[K(S)\otimes K(T)^\vee]^G\otimes A(T).$$
    Noting that a basis morphism $\Hom_{\wBr_n}(S,T)$ is given by a bipartite matching of $(S_1\sqcup T_2, S_2\sqcup T_1)$, Proposition \ref{prop:invariantsofK} tells us that there is a surjection
    $$\Hom_{\wBr_n}(S,T)\twoheadrightarrow [K(S)\otimes K(T)^\vee]^G,$$
    for any $S,T\in\wBr_n$. Furthermore, these are the components of a natural transformation of bifunctors $\wBr_n\times\wBr_n^{\mathrm{op}}\to\bQ\dashmod$ from $\Hom_{\wBr_n}(-,-)$ to $[K\otimes K^\vee]^G$ . Thus we have a surjection
    \begin{align*}
        \kappa:i_*(A)(S)=\int^{T\in\dwBr}\Hom_{\wBr_n}(S,T)\otimes A(T)\to\int^{T\in\dwBr_n}[K(S)\otimes K(T)^\vee]^G\otimes A(T).
    \end{align*}
    Unraveling the adjunctions used to define $\phi$ shows that $\phi^{\wBr_n}_S$ is precisely the composition 
    $$i_*(A)(S)\overset{\kappa}{\longrightarrow}\int^{T\in\dwBr}[K(S)\otimes K(T)^\vee]^G\otimes A(T)\overset{\Delta(\phi)}{\longrightarrow}[i^*K(S)\otimes B^{\alg}]^G.$$
    Since this composition is an isomorphism by assumption, it follows that $\kappa$ must be injective and thus an isomorphism. If this is the case, so must $\Delta(\phi)$. 

    Since $\Delta(\phi)$ is an isomorphism, it follows by Lemma \ref{lemma:testisos} that both the kernel and cokernel of $\phi$ are in $\Rep(G)_K^\circ$. By definition, this is the category of representations $V$ such for any $S\in\dwBr$, we have $[V\otimes K(S)]^G=\Hom_{\Rep(G)}(1_{\Rep(G)},V\otimes K(S))=0$. If $V$ contains an algebraic subrepresentation, it follows by Proposition \ref{prop:invariantsofK} and the fact that every irreducible algebraic representation of $G$ is a summand of $K(S)$, for some $S\in\dwBr$, that $V\in\Rep(G)_K^\circ$ precisely when it has no algebraic subrepresentation. The kernel of $\phi$ is a subrepresentation of $i^*(K^\vee)\otimes^{\dwBr}(1_{\Rep(G)}\tensoredover A)$, which is algebraic, since it is a quotient of a direct sum of representations of the form $K(S)\otimes A(S)$, where $K(S)$ is algebraic and $A(S)$ is trivial. Since $\ker(\phi)\in\Rep(G)_K^\circ$, it thus follows that it must be trivial, so $\phi$ is injective. Since the image of $\phi$ is also an algebraic representation and the cokernel lies in $\Rep(G)_K^\circ$, meaning that every algebraic subrepresentation of $B$ is in the image of $\phi$, it follows that $\mathrm{im}(\phi)=B^{\alg}$. 
    
    To prove the final part of the proposition, we use that by Proposition \ref{prop:KanextensionK}, we have a natural transformation $\psi:i_*K^\circ\to K$, such that $\psi_S$ is surjective and is an isomorphism for $n\gg |S|$. Since $[-]^G$ is exact, it follows that we get a natural transformation $\psi':i_*[K^\circ\otimes B]^G\to [K\otimes B]^G$ such that $\psi'_S$ is also surjective and an isomorphism for $n\gg |S|$. If $S$ is such that $n\gg|S|$ we thus have an isomorphism
    $\phi^{\wBr_n}_S:i_*(A)(S)\to i_*[K^\circ(S)\otimes B]^G$, and thus an isomorphism $A(S)\to[K^\circ(S)\otimes B]^G$. For $S=([p],[q])$, with $n\gg p+q$, we thus have
    $$A(p,q)\cong [K^\circ(p,q)\otimes B]^G=[K_{p,q}^\circ\otimes B]^G$$
    as $\Sigma_p\otimes\Sigma_q$-representations and so it follows by Theorem \ref{thm:schurweyl} that
    \begin{align}\label{eq:applyingSchurWeyl}
        A(p,q)\cong \bigoplus_{|(\lambda,\mu)|=(p,q)}S^\lambda\otimes S^\mu\otimes [V_{\lambda,\mu}\otimes B]^G.
    \end{align}
    Applying the functor $\Hom_{\Sigma_p\otimes \Sigma_q}(S^\lambda\otimes S^\mu,-)$ to both sides now implies that
    $$\Hom_{\Sigma_p\otimes \Sigma_q}(S^\lambda\otimes S^\mu,A(p,q))\cong [V_{\lambda,\mu}\otimes B]^G\cong \Hom_G(V_{\lambda,\mu}^\vee, B)\cong\Hom_G(V_{\mu,\lambda}, B)$$
    Since the dimension of the first term is the multiplicity of $S^\lambda\otimes S^\mu$ in $A(p,q)$ and the dimension of the final term is the multiplicity of $V_{\mu,\lambda}$ in $B$. 
\end{proof}

We also have the following proposition, analogous to \cite[Proposition 2.17]{KupersRandal-Williams}, and which will similarly be used to determine the ring structure of Theorem \ref{theoremA}:

\begin{proposition}\label{prop:wBr2}
    For $A\in(\bQ\dashmod)^{\dwBr}$ of finite length there is a natural transformation
    $$\psi^{\wBr_n}:i_*(A)\to [K\otimes(i^*(K^\vee)\otimes^{\dwBr}(1_{\mathrm{Rep}(G)} \tensoredover A))]^G$$
    of functors in $(\bQ\dashmod)^{\wBr_n}$, which is an epimorphism. Furthermore, since $A$ is assumed to have finite length, there exists an $N\ge 0$, such that $A(T)=0$ for all $T\in\dwBr$ with $|T|\ge N$. We have that $\psi^{\wBr_n}_S$ is an isomorphism for all $S\in\wBr_n$ with $|S|\le n-N-2$.
\end{proposition}

\begin{proof}
    Once again, we closely follow the proof of \cite[Proposition 2.17]{KupersRandal-Williams} and start by constructing the map $\psi^{\wBr_n}$. For any $S\in\dwBr$, let $S'=(S_1\sqcup S_2,S_1\sqcup S_2)$. There is then a canonical bipartite matching $m:=\{(s,s)\mid s\in S_1\sqcup S_2\}$ in $\mathrm{M}(S')$. Since $K(S)\otimes K^\vee(S)\cong K(S')$ as $G$-representations, we thus have a linear map $\omega_m:\bQ\to(K(S)\otimes K^\vee(S))^G$. We thus get a linear map $\psi^{\dwBr}$ given by the composition
    \[\begin{tikzcd}
        A(S)\arrow[rr,"\omega_m\otimes\id"]&& \left[K(S)\otimes K^\vee(S)\right]^G\otimes A(S) \arrow[r,"\mathrm{inc}"] &\left[\int^{T\in\dwBr} K(S)\otimes K^\vee(T)\otimes A(T)\right]^G
    \end{tikzcd}\]
    where we note that the target is precisely $[i^*K(S)\otimes (i^* K^\vee\otimes^{\dwBr}(1_{\Rep(G)}\tensoredover A))]^G$. These are the components of a natural transformation
    $$\psi^{\dwBr}:A\to[i^*K\otimes (i^* K^\vee\otimes^{\dwBr}(1_{\Rep(G)}\tensoredover A))]^G,$$
    and we define $\psi^{\wBr_n}$ to be its adjoint under the adjunction between $i_*$ and $i^*$.

    As in the proof of the previous proposition, we have a natural transformation 
    $$\kappa:\Hom_{\wBr_n}(-,-)\to [K\otimes K^\vee]^G$$
    of two variables, such that for any $S,T\in\wBr_n$, the map $\kappa_{S,T}$ is surjective. Using the coend formula for left Kan extension, we have for $S\in\wBr_n$ that
    $$\psi^{\wBr_n}_{S}:i_*(A)(S)=\int^{T\in\dwBr}\Hom_{\wBr_n}(T,S)\tensoredover A(T)\to \int^{T\in\dwBr}[K(S)\otimes K^\vee(T)]^G\otimes A(T),$$
    which is precisely the map induced on coends by $\kappa$ and thus also surjective. Thus $\psi^{\wBr_n}$ is an epimorphism. 

    If $|T|\ge N$, we have by assumption that $A(T)=0$, so then the map
    $$\kappa_{S,T}\otimes\id:\Hom_{\wBr_n}(S,T)\otimes A(T)\to [K(S)\otimes K^\vee(T)]^G\otimes A(T)$$
    is trivially an isomorphism. Furthermore, by Proposition \ref{prop:invariantsofK}, $\kappa_{S,T}$ is an isomorphism whenever $n\ge |S|+|T|+3$. Thus $\psi^{\wBr_n}$ is an isomorphism whenever $N\le|T|$ or $|S|\le n-|T|-3$, so in particular when $|S|\le n-N-2$, which is what we wanted to show.
\end{proof}

\section{Stable cohomology with twisted coefficients}\label{sec:twistedcoeffs}

\noindent The main input for the proof of Theorem 1 is the stable cohomology of $\Aut(F_n)$ with coefficients in the representations $K(S):=H(n)^{\otimes S_1}\otimes H^\vee(n)^{\otimes S_2}$, for all pairs of finite sets  $S=(S_1,S_2)$. These were computed by the author in \cite{Lindell2022stable}. In order to state the result, let us define the following graded representations of $\Sigma_{S_1}\times\Sigma_{S_2}$:

\begin{definition}
    For $S=(S_1,S_2)$ a pair of finite sets, we make the following definitions:\begin{itemize}
        \item Let $\PP(S)=\PP(S_1,S_2)$ be the graded vector space, concentrated in degree $|S_1|-|S_2|$, generated by partitions of $S_1$ with at least $|S_2|$ parts, where $|S_2|$ of the parts labeled by different elements of $S_2$ and the remaining parts are unlabeled (we can think of these as being labeled by zero). The $\Sigma_{S_1}\times\Sigma_{S_2}$-action is given by permuting elements and labels. If $\{P_1,P_2,\ldots,P_l\}$ is a partition of $S_1$, we will write
        $$\{(P_i,l_i)\}_{1\le i\le l}\in \PP(S) $$
        for the labeled partition where $P_i$ is labeled by $l_i$, which is either zero or an element of $S_2$.
        \item Let $\det(S):=\det(\bQ^{S_1})\otimes\det(\bQ^{S_2})$, where $\det(\bQ^{S_i}):=\exterior{S_i}\bQ^{S_i}$ is the sign representation of $\Sigma_{S_i}$.
    \end{itemize} 
\end{definition}

With this definition we can state the main theorem of \cite{Lindell2022stable} as follows:

\begin{theorem}\label{thm:twistedcoeffs}
    Let $p,q\ge 0$. In degrees $2*\le n-p-q-3$, we have \begin{equation}\label{eq:stablecohomology}
\PP(p,q)\otimes\det(p,q)\cong H^*(\Aut(F_n),K_{p,q}(n))
    \end{equation}
    as $\Sigma_p\times\Sigma_q$-representations.
\end{theorem}

We describe the isomorphism explicitly in the following sections. After this, we show how both source and target naturally define the components of functors $\wBr_n\to\bQ\dashmod$ and how the isomorphism in Theorem \ref{thm:twistedcoeffs} defines the components of a natural isomorphism between these functors.

\subsection{Twisted cohomology classes} To explicitly describe the isomorphism (\ref{eq:stablecohomology}), we start by recalling the definition of certain cohomology classes introduced by Kawazumi in \cite{KawazumiMagnusExpansions}. 

If we let $A_n^2:=\Aut(F_n)\ltimes F_n$, we have a crossed group homomorphism $k_0:A_n^2\to H(n)$ given by $(\phi,\gamma)\mapsto [\gamma]$, which thus defines a cohomology class $[k_0]\in H^1(A_n^2,H(n))$. For any $p\ge 1$, the short exact sequence
$$1\to F_n\to A_n^2\to\Aut(F_n)\to 1$$
gives us a Gysin map $\pi_!:H^{p}(A_n^2,H(n)^{\otimes p})\to H^{p-1}(\Aut(F_n), H(n)^{\otimes p}\otimes H^\vee(n))$, so for each $p\ge 1$ we can define a class
$$h_{p,1}:=\pi_!([k_0]^{\otimes p})\in H^{p-1}(\Aut(F_n),H(n)^{\otimes p}\otimes H^\vee(n)).$$
Here we note that $h_{1,1}=[\id_{H(n)}]\in H^0(\Aut(F_n),\Hom_\bQ(H(n),H(n)))$.

Furthermore, with $\lambda_{1,1}:H(n)^{\otimes(p+1)}\otimes H^\vee(n)\to H(n)^{\otimes p}$ being the contraction map, we define for each $p\ge 0$ a class
$$h_{p,0}:=(\lambda_{1,1})_*(h_{p+1,1})\in H^p(\Aut(F_n),H(n)^{\otimes p}).$$
We can note that since $h_{1,1}=[\id_{H(n)}]$, it follows that $h_{0,0}=n\in H^0(\Aut(F_n),\bQ)=\bQ$.
\begin{remark}
    In \cite{KawazumiMagnusExpansions}, the class $h_{p,1}$ is denoted $h_{p-1}$ and the class $h_{p,0}$ is denoted by $\bar{h}_p$. In Kawazumi's notation, the index corresponds to the cohomological degree, whereas in our case, $h_{p,q}\in H^{p-q}(\Aut(F_n),K_{p,q}(n))$, for $p-q\ge 0$ and $q=0,1$. We use this different notation since it will make our indexing easier going forward.
\end{remark}

\subsection{The wheeled PROP-structure} Kawazumi and Vespa \cite{KawazumiVespa} recently showed that the stable cohomology groups of $\Aut(F_n)$ with the coefficients $K_{p,q}(n)$ assemble to something called a \textit{wheeled PROP}. The definition of this structure is a bit involved, so we refer the reader to \cite{MarklMerkulovShadrin} for details and only give a partial definition. 

\begin{definition}
    A wheeled PROP $\cC$ in $\Gr(\bQ\dashmod)$ is a collection $\{\cC(p,q)\}_{p,q\ge 0}$ of graded vector spaces, with a commuting left $\Sigma_p$-action and right $\Sigma_q$-action, together with equivariant structure maps\begin{enumerate}[(1)]
        \item $\otimes:\cC(p_1,q_1)\otimes\cC(p_2,q_2)\to\cC(p_1+p_2,q_1+q_2)$, called \textit{horizontal composition},
        \item $\circ:\cC(p,q)\otimes\cC(q,r)\to\cC(p,r)$ called \textit{vertical composition} and
        \item for $p,q\ge 1$ and each $1\le i\le p$, $1\le j\le q$ a map $\xi_{i,j}:\cC(p,q)\to\cC(p-1,q-1)$ called \textit{contraction},
    \end{enumerate}
    together with a number of compatibility conditions, detailed in \cite{MarklMerkulovShadrin}. We refer to the space $\cC(p,q)$ as the \textit{space of operations of} \textit{biarity} $(p,q)$.
\end{definition}

\begin{remark}\label{rmk:compositioncontraction}
    One of the compatibility conditions of wheeled PROPs is that the vertical composition maps are in fact determined by the horizontal composition maps and the contraction maps, in the sense that the diagram
    \[\begin{tikzcd}
        \cC(p,q)\otimes\cC(q,r)\arrow[r,"\otimes"]\arrow[rd,"\circ"]&\cC(p+q,q+r)\arrow[d,"\xi_{1,p+1}^{\circ q}"]\\
        &\cC(p,r)
    \end{tikzcd}\]
    is commutative, where $\xi^{\circ q}_{p+1,1}$ denotes the $q$-fold composition.
\end{remark}

The collection $\{K_{p,q}(n)\}_{p,q\ge 0}$ naturally forms a wheeled PROP, with the composition maps induced by the literal composition maps on the spaces $\Hom_\bQ(H(n)^{\otimes q},H(n)^{\otimes p})$ and the contraction maps induced by the duality pairing $\lambda:H(n)\otimes H^\vee(n)\to\bQ$. By using the cup product, this induces corresponding maps on the collection $\{H^*(\Aut(F_n),K_{p,q}(n))\}_{p,q\ge 0}$. If we let $\mathcal{H}(p,q):=\lim_{n}H^*(\Aut(F_n),K_{p,q}(n))$, it was proven by Kawazumi and Vespa in \cite{KawazumiVespa} that this makes the collection $\mathcal{H}:=\{\mathcal{H}(p,q)\}_{p,q\ge 0}$ into a wheeled PROP. 

In \cite{KawazumiMagnusExpansions}, it was proven that the class $h_{2,1}$ satisfies the relations
\begin{align}
    (1\otimes h_{2,1})\circ h_{2,1}+(h_{2,1}\otimes 1)\circ h_{2,1}&=0\label{eq:kawazumi-rel1}\\
    h_{2,1}+\tau\cdot h_{2,1}&=0,\label{eq:kawazumi-rel2}
\end{align}
where $\otimes$ and $\circ$ denotes the wheeled PROP-composition maps, $1\in H^0(\Aut(F_n),K_{1,1}(n))$ is the cohomology class of the identity in $K_{1,1}(n)\cong\Hom_{\bQ}(H(n),H(n))$ and $\tau\in\Sigma_2$ denotes the transposition.

Furthermore, the classes $h_{p,1}$ and $h_{p,0}$, for $p\ge 1$, are generated by $h_{2,1}$ using the wheeled PROP structure. First of all, we have $h_{1,1}=1$ and by \cite[Theorem 4.1]{KawazumiMagnusExpansions}, we have that for $p\ge 2$
\begin{align}\label{eq:hp-from-h1}
    h_{p,1}=(h_{2,1}\otimes 1^{\otimes (p-2)})\circ (h_{2,1}\otimes 1^{\otimes (p-3)})\circ\cdots\circ h_{2,1}.
\end{align}
Since $h_{p,0}=(\lambda_{1,1})_*(h_{p+1},1)=\xi_{1,1}(h_{p+1,1})$, these classes are also generated by $h_{2,1}$ under the wheeled PROP-structure. The contraction relations between these classes will be important when considering functoriality on the walled Brauer category, so we summarize them in the following proposition:
\begin{proposition}\label{prop:contractionrelations}
    For any $p\ge 1$, $p
    \ge 1$ and $i=0,1$, we have 
    \begin{align*}
        \xi_{1,1}(h_{p,1})&=h_{p-1,0},\\
        \xi_{p+1,1}(h_{p,1}h_{p',i})&=h_{p+p'-1,i}.
    \end{align*}
\end{proposition}

\begin{proof}
    The first relation holds for all $p\ge 1$ by definition (noting that $h_{0,0}=n\in\bQ$). For the second relation, first assume that $p=1$. Then we have $h_{p,1}=h_{1,1}=1$, so $\xi_{2,1}(h_{1,1}h_{p',i})=h_{p',i}$ holds trivially. Now suppose $p\ge 2$ and $i=1$. By using Remark \ref{rmk:compositioncontraction} and (\ref{eq:hp-from-h1}), we get
    \begin{align*}
        \xi_{p+1,1}(h_{p,1}h_{p',1})&=(h_{p,1}\otimes 1^{\otimes (p'-1)})\circ h_{p',1}\\
        &=(((h_{2,1}\otimes 1^{\otimes (p-2)})\circ (h_{2,1}\otimes 1^{\otimes (p-3)})\circ\cdots\circ h_{2,1})\otimes 1^{\otimes (p'-1)})\\
        &\ \ \  \circ ((h_{2,1}\otimes 1^{\otimes (p'-2)})\circ (h_{2,1}\otimes 1^{\otimes (p'-3)})\circ\cdots\circ h_{2,1})\\
        &=(h_{2,1}\otimes 1^{\otimes (p+p'-3)})\circ(h_{2,1}\otimes1^{\otimes (p+p'-4)})\circ\cdots\circ h_{2,1}\\
        &=h_{p+p'-1,1}.
    \end{align*}
    Using the first case and the relation (\ref{eq:kawazumi-rel2}) we also have
    \begin{align*}
    \xi_{p+1,1}(h_{p,1}h_{p',0})&=\xi_{p,1}(h_{p,1}\xi_{1,1}(h_{p'+1,1}))\\
    &=\xi_{p+1,1}(\xi_{p+1,2}(h_{p,1}h_{p'+1,1}))\\
    &=\xi_{p+1,1}(\xi_{p+2,1}(h_{p,1}h_{p'+1,1}))\\
    &=\xi_{p+1,1}(-\xi_{p+1,1}(h_{p,1}h_{p'+1,1}))\\
    &=\xi_{p+1,1}(-h_{p+p',1})\\
    &=\xi_{1,1}(h_{p+p',1})\\
    &=h_{p+p'-1,0},
    \end{align*}
    which finishes the proof.
\end{proof}

We will write $\cK:=\langle h_{2,1}\rangle\subseteq \cH$ for the wheeled sub-PROP of $\cH$ generated by the class $h_{2,1}$. One of the main results of \cite{KawazumiVespa} may be stated as follows:

\begin{theorem}[Kawazumi-Vespa]
    The wheeled PROP $\cK$ is isomorphic to the wheeled PROP generated by one operation in biarity $(2,1)$ and degree 1, satisfying the relations (\ref{eq:kawazumi-rel1}) and (\ref{eq:kawazumi-rel2}), with the isomorphism defined by $\mu\mapsto h_{2,1}$.
\end{theorem}

We proceed by constructing a map 
$$\PP(S)\otimes\det(S)\to H^*(\Aut(F_n),K_S(n)),$$
for any pair $S=(S_1,S_2)$ of finite sets, following the example of \cite[Section 3.3]{KupersRandal-Williams}. 

Suppose that $p\ge q$ and let $(\lambda,k):=((\lambda_1,k_1),\ldots,(\lambda_l,k_l))$ be an $l$-tuple of pairs of integers such that $(\lambda_1\ge\lambda_2\ge\cdots\ge\lambda_l)$ is a partition of weight $|\lambda|=p$ and $k_i=0,1$ for each $1\le i\le l$ are such that $k_1+\cdots+k_l=q$. To this data we can associate the cohomology class
$$\kappa(\lambda,k):=h_{\lambda_1,k_1}\cdots h_{\lambda_l,k_l}\in H^*(\Aut(F_n),K_{p,q}(n))\otimes\det(p,q).$$
If we let $k_i':=k_1+k_2+\cdots+k_i$ if $k_i=1$ and $k_i'=0$ otherwise, we can also associate a labeled partition of the set $\{1,\ldots,p\}$ to $(\lambda,k)$ by
\begin{align*}
    P_{\lambda,k}:=\{&(\{1,\ldots,\lambda_1\},k_1'),(\{\lambda_1+1,\ldots,\lambda_1+\lambda_2\},k_2'),\ldots, (\{\lambda_1+\cdots+\lambda_{l-1}+1,\ldots,p\},k_l')\}
\end{align*}
in $\PP(p,q)$. We call this the \textit{standard labeled partition} associated to $(\lambda,k)$.

Now suppose $S=(S_1,S_2)$ is a pair of finite sets and that $P\in \PP(S)$ is any labeled partition. Let $p_1\ge p_2\ge\cdots\ge p_l\ge 0$ be the sizes of the parts of $P$. Let us write $\lambda_P:=(p_1\ge p_2\ge \cdots\ge p_k)$. Then there exists a bijection $f:S_1\to [p]$ sending $P$ to the partition 
$$\{\{1,\ldots,p_1\},\{p_1+1,\ldots,p_1+p_2\},\ldots,\{p_1+\cdots p_{l-1}+1,\ldots, p_1+\cdots+p_l\}\}$$
We can now order the elements of $S_2$ according to the order of the images in $f(S_1)$ of the parts they label. This gives us a bijection $g:S_2\to [q]$ (which thus depends on $f$) such that $f\times g$ sends $P$ to a standard labeled partition $P_{\lambda_P,k_P}$ of $[p]$ with labels in $[q]$. We will therefore define a map
$$\kappa_{S}:\PP(S)\to H^*(\Aut(F_n),K_S(n))\otimes\det(S)$$
by
\begin{equation}\label{eq:twistedclassdef}
    \kappa_{S}(P):=(f^{-1}\times g^{-1})_*(\kappa(\lambda_P,k_P)\otimes(e_1\wedge\cdots\wedge e_p)\otimes(e_1\wedge\cdots\wedge e_q)).
\end{equation}

\begin{proposition}\label{prop:kappa-welldefined}
    The map $\kappa_S$ is well-defined, $\Sigma_{S_1}\times\Sigma_{S_2}$-equivariant and injective in a stable range of degrees.
\end{proposition}

\begin{proof}
That the map is well-defined is proven using essentially the same argument as in \cite[Lemma 3.2]{KupersRandal-Williams}: If $f':S_1\to[p]$ and $g':S_2\to[q]$ also are bijections such that $((f')^{-1}\times (g')^{-1})(P)=P_{\lambda_P,k_P}$, it follows that $(f'\times g')\otimes (f^{-1}\times g^{-1}):[p]\times [q]\to[p]\times[q]$ preserves the labeled partition $P_{\lambda_P,k_P}$. The subgroup of $\Sigma_p\times\Sigma_q$ preserving  $P_{\lambda_P,k_P}$ is generated by permutations of the individual parts $Q_i$ of this partition, permutations of different unlabeled parts of the same size and permutations of labeled parts of the same size, together with a corresponding permutation of the labels. A permutation $\sigma$ of $Q_i=\{p_1+\cdots+p_{i-1}+1,\cdots,p_1+\cdots p_i\}$ acts on $h_{p_i,k_i'}$ by permuting the factors of $[k_0]^{\otimes p_i}$ and since $[k_0]$ is of degree one, this has the effect of changing the sign by $\sgn(\sigma)$. But this is also the effect on $\det(p)$, so $\sigma$ acts trivially on $\kappa(\lambda_P,k_P)$.

The group of permutations of the set $\{Q_i\in P_{\lambda_P,k_P}\mid |Q_i|=j,\text{ }Q_i\text{ unlabeled}\}$ is generated by transpositions of pairs of unlabeled parts. Such a permutation $\sigma$ acts on the first factor of $\cK(p,q)\otimes\det(p,q)$ by transposing the first and last factor in a product of the form $h_{j,0}ch_{j,0}$, where $c$ is some cohomology class, i.e.\ by the sign $(-1)^{(|c|+j)j+|c|j}=(-1)^{j^2}=(-1)^j$. On $\det(p,q)$ it acts by the sign of $j$ transpositions, i.e.\ by $(-1)^j$. Thus $\sigma$ acts trivially on $\kappa(\lambda_p,k_P)$.

Similarly, the subgroup of $\Sigma_p\times\Sigma_q$ which permutes labeled parts of the same size and their labels accordingly is generated by transpositions of labeled parts. Such a product $\sigma_1\times\sigma_2$ of permutations acts on the first factor of $\cK(p,q)\otimes\det(p,q)$ by transposing the first and last factor in a product $h_{j,1}ch_{j,1}$ and thus acts by the sign $(-1)^{(j-1)(|c|+j-1)+(j-1)|c|}$ $=(-1)^{(j-1)^2}=(-1)^{j-1}$. Since $\sigma_1$ consists of $j$ transpositions and $\sigma_2$ by one transposition, $\sigma_1\times\sigma_2$ acts on  $\det(p,q)$ by $(-1)^{j+1}$, so it acts trivially on $\kappa(\lambda_P,k_P)$.

The map is equivariant essentially by construction, but let us show it for completeness. For brevity, let us for any $k$ write $v_k:=e_1\wedge\cdots\wedge e_k$ for the generator in $\det(\bQ^k)$. Let $\sigma_1\in\Sigma_{S_1}$ and $\sigma_2\in\Sigma_2$. If $f,g$ are as in the construction of $\kappa_S$, it follows that $(f\circ\sigma_1^{-1})\times(g\circ\sigma_2^{-1})$ sends the labeled partition $(\sigma_1\times\sigma_2)\cdot P$ to $P_{\lambda_P,k_P}$. We thus have
\begin{align*}
    \kappa_S((\sigma_1\times\sigma_2)\cdot P)&=\left(f\circ\sigma_1^{-1})^{-1}\times(g\circ\sigma_2^{-1})^{-1}\right)_*(\kappa(\lambda_P,k_P)\otimes v_p\otimes v_q)\\
    &=(\sigma_1\circ f^{-1})\times(\sigma_2\circ g^{-1})_*(\kappa_{\lambda_P,k_P}\otimes v_p\otimes v_q)\\
    &=\left((\sigma_1\times\sigma_2)\circ (f^{-1}\times g^{-1})\right)_*(\kappa_{\lambda_P,\kappa_P}\otimes v_p\otimes v_q)\\
    &=(\sigma_1\times\sigma_2)_*\cdot \kappa_S(P).
\end{align*}
To show injectivity in the stable range, we start by noting that choosing any pair of bijections $f:S_1\to [p]$ and $g:S_2\to[q]$, gives us an isomorphism $H^*(\Aut(F_n),K_S(n))\to \cH(p,q)$ and it is sufficient to show that the precomposition of this isomorphism with $\kappa_S$ is injective. This composition lands in $\cK(p,q)$ by construction. Furthermore, every class of the form $\kappa(\lambda,k)$ is in the image and these generate $\cK(p,q)$ as a $\Sigma_p\times\Sigma_q$-representation. The composite map is $\Sigma_p\times\Sigma_q$-equivariant, by considering the source as a representation via the isomorphism $\Sigma_{S_1}\times\Sigma_{S_2}\to\Sigma_p\times\Sigma_q$ induced by $f\times g$. Thus it follows that the map surjects onto $\cK(p,q)$. However, a simple dimension counting argument using that the only relations in $\cK$ are (\ref{eq:kawazumi-rel1}) and (\ref{eq:kawazumi-rel2}) shows that 
$$\dim(\PP(S))=\dim(\PP(p,q))\le\dim(\cK(p,q)),$$
so the map is injective.
\end{proof}

By adjunction, we get an equivariant map
$$\Phi_{S}:\PP(S)\otimes\det(S)\to H^*(\Aut(F_n), K_{S}(n)),$$
which is also injective in the stable range. Since these representations are abstractly isomorphic in a stable range, by Theorem \ref{thm:twistedcoeffs}, we can conclude the following:

\begin{corollary}
    For any pair $S=(S_1,S_2)$ of finite sets, the map 
    $$\Phi_{S}:\PP(S)\otimes\det(S)\to H^*(\Aut(F_n),K_{S}(n))$$ is an isomorphism in the range $2*\le n-|S_1|-|S_2|-3$.
\end{corollary}

\subsection{Functoriality}\label{subsec:Functoriality} Our next goal is to understand $H^*(\Aut(F_n),K_S(n))$ as the components of a functor on the category $\wBr_n$. The functor $K:\wBr_n\to\Rep(\GL_n(\bZ))$, which was introduced in Section \ref{sec:walledBrauer}, factors through $\Rep(\Aut(F_n))$, so let us by abuse of notation also write $K$ for the functor to this category. By postcomposition with the functor $H^*(\Aut(F_n),-)$ we thus get a functor
$$H^*(\Aut(F_n),K):\wBr_n\to\Gr(\bQ\dashmod).$$
Let us describe this functor in terms of $\PP(S)\otimes\det(S)$.

We will define the functor $\PP\otimes\det:\wBr_n\to \bQ\dashmod$ by first describing the tensor factors separately. We have already defined $\PP(S)$ for $S\in\wBr_n$, so let us describe what is does on morphisms:
\begin{enumerate}[(1)]
    \item A morphism of the form $(f,g,\varnothing,\varnothing)$, where $f:S_1\to T_1$ and $g:S_2\to T_2$ are bijections, is sent to the obvious map $\PP(S)\to\PP(T)$ given by permuting partitions and labels.
    \item A morphism of the form $(\id_{S_1},\id_{S_2}, \varnothing, (x,y)):(S_1,S_2)\to (S_1\sqcup\{x\},S_2\sqcup\{y\})$ is sent to the map $\PP(S_1,S_2\{y\})\to\PP(S_1\sqcup\{x\},S_2\sqcup\{y\})$ which ads the labeled part $(x;y)$,
    \item A morphism of the form $(\id_{S_1\setminus \{x\}},\id_{S_2\setminus\{y\}}, (x,y), \varnothing):(S_1,S_2)\to(S_1\setminus \{x\},S_2\setminus \{y\})$ is sent to the map $\PP(S_1,S_2)\to\PP(S_1\setminus\{x\},S_2\setminus\{y\})$ defined as follows:\begin{enumerate}[(a)]
        \item If $(\{P_i,y_i\})$ is a labeled partition such that $P_i=\{x\}$ and $y_i=y$ for some $i$, we remove this part and multiply the labeled partition by $n$.
        \item If $(\{P_i,y_i\})$ is a labeled partition such that $x\in P_i$ and $y_i=y$ for some $i$, but $P_i\neq \{x\}$, we change this part to the unlabeled part $\{P_i\setminus \{x\},0\}$ and
        \item If $(\{P_i,y_i\})$ is a labeled partition such that $x\in P_i$ and $y_j=y$, for $i\neq j$ we merge the parts into the labeled part $\{(P_i\setminus\{x\})\sqcup P_j,y_i\}$.
    \end{enumerate}
\end{enumerate}
Now let us define the second tensor factor: we let $\det(S)=\det\left(\bQ^{S_1}\right)\otimes \det\left(\bQ^{S_2}\right)$ as above and for a morphism $(f,g,m_S,m_T)$ we define $\det(f,g,m_S,m_T)$ as follows (cf. \cite[Sections 3.5-3.6]{KupersRandal-Williams}): 
\begin{enumerate}[(1)]
    \item A morphism of the form $(f,g,\varnothing,\varnothing)$, where $f:S_1\to T_1$ and $g:S_2\to T_2$ are bijections, is send to the obvious map $\det(S)\to\det(T)$ by permuting tensor factors.
    \item A morphism of the form $(\id_{S_1},\id_{S_2}, \varnothing, (x,y)):(S_1,S_2)\to (S_1\sqcup\{x\},S_2\sqcup\{y\})$ is sent to the map $(-\wedge x)\otimes (-\wedge y)$.
    \item A morphism of the form $(\id_{S_1\setminus \{x\}},\id_{S_2\setminus\{y\}}, (x,y), \varnothing):(S_1,S_2)\to(S_1\setminus \{x\},S_2\setminus \{y\})$ is sent to the inverse of the map described in the previous case.
\end{enumerate}

Since any morphism in $\wBr_n$ can be written as a composition of morphisms of the kinds considered above, this describes the functor $\PP\otimes\det$ completely.

\begin{proposition}
    The maps $\Phi_S:\PP(S)\otimes\det(S)\to H^*(\Aut(F_n),K_S(n))$ define the components of a natural transformation of functors in $(\Gr(\bQ\dashmod))^{\wBr_n}$.
\end{proposition}

\begin{proof}
    We will show that the maps $\kappa_S:\PP(S)\to H^*(\Aut(F_n),K_S(n))\otimes\det(S)$ define a natural transformation. The statement for $\Phi_S$ then follows by adjunction.
    
    It is sufficient to show that for any morphism $F\in\Hom_{\wBr_n}(S,T)$ of the three kinds described in the cases (1)-(3) above, we have a commutative diagram
    \[\begin{tikzcd}
        \PP(S)\arrow[r,"F_*"]\arrow[d,"\Phi_S"]&\PP(T)\arrow[d,"\Phi_T"]\\
        H^*(\Aut(F_n),K(S))\otimes\det(S)\arrow[r,"F_*"]&H^*(\Aut(F_n),K(T))\otimes\det(T)
    \end{tikzcd}\]
    where we write $F_*$ for the morphism given by functoriality, for brevity. For $F=(f,g,\varnothing,\varnothing)$, this is proven in the same way as the equivariance of the map $\Phi_S$.

    Suppose $T=(S_1\sqcup\{x\},S_2\sqcup\{y\})$ and $F=(\id_{S_1},\id_{S_2},\varnothing, (x,y))$, where $|S_1|=p$ and $|S_2|=q$. For any $f:S_1\to[p]$ and $g:S_2\to[q]$, we can extend these to $\bar{f}:S_1\sqcup\{x\}\to[p+1]$ and $\bar{g}:S_2\sqcup\{y\}\to[q+1]$ by sending $x\mapsto p+1$ and $y\mapsto q+1$. As before, let us for brevity write $v_{k}:=e_1\wedge e_2\wedge\cdots\wedge e_k$ for the generator of $\det(\bQ^{k})$, for any $k\ge 0$. If $P\in\PP(S)$, let $f\times g:S_1\times S_2\to [p]\times[q]$ be bijections sending $P$ to a standard labeled partition. Then 
    \begin{align*}
        F_*(\kappa_S(P))&=F_*((f^{-1}\times g^{-1})_*(\kappa(\lambda_P,k_P)\otimes v_p\otimes v_q))\\
        &=\omega_{x,y}\left((f^{-1}\times g^{-1})_*\kappa(\lambda_P,k_P)\right)\otimes(f^{-1}_*(v_p)\wedge x\otimes g^{-1}_*(v_q)\wedge y)\\
        &=(\bar{f}^{-1}\times \bar{g}^{-1})_*(\kappa(\lambda_P,k_P)h_{1,1}\otimes v_{p+1}\otimes v_{q+1})\\
        &=(\bar{f}^{-1}\times\bar{g}^{-1})_*(\kappa(\lambda_{P\sqcup(x;y)},k_{P\sqcup(x;y)})\otimes v_{p+1}\otimes v_{q+1})\\
        &=\kappa_T((P\sqcup(x;y))) \\
        &= \kappa_T(F_*(P)).
    \end{align*}

    Next suppose $F$ is as in (3). If $f:S_1\to[p]$ and $g:S_2\to [q]$ are bijections, this once again induces in a unique way bijections $\bar{f}:S_1\setminus\{x\}\to [p-1]$ and $\bar{g}:S_2\setminus\{y\}\to[q-1]$, by preserving the order of the elements in the image.
    
    First suppose $P\in\PP(S)$ is a labeled partition as in (a). Let $f$ and $g$ be bijections as in the definition of $\kappa_S$. Using the first contraction relation in Proposition \ref{prop:contractionrelations} we get
    \begin{align*}
        F_*(\kappa_S(P))&=F_*((f^{-1}\times g^{-1})_*(\kappa(\lambda_P,k_P)\otimes v_p\otimes v_q))\\
        &=\lambda_{x,y}\left((f^{-1}\times g^{-1})_*\kappa(\lambda_P,k_P)\right)\otimes (\bar{f}^{-1}_*v_{p-1}\otimes\bar{g}^{-1}_*v_{q-1})\\
        &=(\bar{f}^{-1}\times \bar{g}^{-1})_*(\lambda_{f(x),g(y)}\kappa(\lambda_P,k_P))\otimes (\bar{f}^{-1}_*v_{p-1}\otimes\bar{g}^{-1}_*v_{q-1})\\
        &=(\bar{f}^{-1}\times \bar{g}^{-1})_*(n\cdot\kappa(\lambda_{P\setminus(x;y)},k_{P\setminus(x;y)})\otimes (\bar{f}^{-1}_*v_{p-1}\otimes\bar{g}^{-1}_*v_{q-1})\\
        &=\kappa_S(n\cdot P\setminus(x;y))\\
        &=\kappa_S(F_*(P)).
    \end{align*}
The case (b) is almost by an almost identical calculation using the same contraction relation, whereas the case (c) follows from the second contraction relation in Proposition \ref{prop:contractionrelations} by a similar calculation.

 \end{proof}

Next, we define a functor $\PP':\dwBr\to\bQ\dashmod$ by letting $\PP'(S)\subset\PP(S)$ be generated by the set of labeled partitions in $\PP(S)$ with no labeled parts of size one. On morphisms, the functor is defined as $\PP$. These two functors are related as follows:

\begin{proposition}
    There is a natural isomorphism $i_*(\PP')\to\PP$ of functors in $(\bQ\dashmod)^{\wBr_n}$.
\end{proposition}

\begin{proof}
    We will define the components of this natural transformation by defining it on each summand of the decomposition of $i_*(\PP')$ from (\ref{eq:KanFormula}), so recall that this states that 
    $$i_*(\PP')(S)\cong\bigoplus_{I\subseteq S}\Hom_{\wBr_n}(\varnothing, S\setminus I)\otimes \PP'(I).$$
    For any $S\in\wBr_n$ and any $I\subseteq S$, there is a linear map $\Hom_{\wBr_n}(\varnothing, S\setminus I)\to\PP(S\setminus I)$, which in the case where $|S_1\setminus I_1|\neq |S_2\setminus I_2|$ is just the zero map, since the source is zero in this case, and otherwise is given by
    $$\{(x_1,y_1),\ldots,(x_k,y_k)\}\mapsto \{(x_1;y_1),(x_2;y_2),\ldots,(x_k;y_k)\}.$$
    There is also a linear map $\PP(S\setminus I)\otimes\PP'(I)\to\PP(S)$, given by merging labeled partitions. We thus get a composite map
    $$i_*(\PP')(S)=\bigoplus_{I\subseteq S}\Hom_{\wBr_n}(\varnothing,S\setminus I)\otimes\PP'(I)\to \bigoplus_{I\subseteq S}\PP(S\setminus I)\otimes\PP'(I)\to \PP(S).$$
    Since the left Kan extension precisely adds all those labeled partitions with some labeled parts of size one, this is an isomorphism.

    Showing that this is indeed a natural transformation  is a straightforward verification, done by comparing the description of the functor $i_*(\PP')$ on the ``generating'' morphisms in $\wBr_n$ in the general description in Section \ref{subsec:leftKanextension} and the similar description in the definition of $\PP$.
\end{proof}

\subsection{Multiplicativity} The functor $H^*(\Aut(F_n),K)$ is lax symmetric monoidal (i.e.\ a commutative ring object in the category $(\bQ\dashmod)^{\wBr_n}$), with the monoidality given by the cup product
$$H^*(\Aut(F_n),K(S))\otimes H^*(\Aut(F_n),K(T))\to H^*(\Aut(F_n),K(S\sqcup T)),$$
for $S,T\in\wBr_n$. The functor $\PP\otimes\det:\wBr_n\to\bQ\dashmod$ is also lax symmetric monoidal, with the monoidality 
$$(\PP(S)\otimes\det(S))\otimes(\PP(T)\otimes\det(T))\to\PP(S\sqcup T)\otimes\det(S\sqcup T)$$
given by disjoint union of labeled partitions and exterior product. This defines a lax symmetric monoidality for the functor $\PP'\otimes\det:\dwBr\to\bQ\dashmod$ in the same way. 
\begin{proposition}
    The natural isomorphism
$$\Phi:\PP\otimes\det\to H^*(\Aut(F_n),K)$$
is compatible with the monoidalities of these functors.
\end{proposition}

\begin{proof}
    Once again, it is sufficient, and somewhat simpler, to verify this statement for the natural isomorphism $\kappa:\PP\to H^*(\Aut(F_n),K)\otimes\det$. Let $S,T\in\wBr_n$ be such that $|S_1|=p_1$, $|S_2|=p_2$, $|T_1|=q_1$, $|T_2|=q_2$. For $P\in\PP(S)$ and $Q\in\PP(T)$, let $f_i:S_i\to [p_i]$ and $g_i:T_i\to[q_i]$ be bijections such that $f_1\times f_2$ sends $P$ to a standard labeled partition $P_{\lambda_P,k_P}\in\PP([p_1],[p_2])$, and $g_1\times g_2$ sends $Q$ to a standard labeled partition $P_{\lambda_{Q},k_{Q}}\in\PP([q_1],[q_2])$. Let us write $\lambda_P=(\lambda_1\ge \cdots\ge \lambda_{l})$, $k_P=(k_1,\ldots,k_l)$, $\lambda_{Q}=(\lambda'_1\ge\cdots\ge\lambda'_r)$ and $k_{Q}=(k_1',\ldots,k_r')$. 
    
    Now let $\bar{g}_i:T_i\to\{p_i+1,p_i+2,\ldots,p_i+q_i\}$ be defined by $\bar{g}_i(t)=g_i(t)+p_i$. This means that $(f_1\sqcup \bar{g}_1)\times(f_2\sqcup\bar{g}_2)$ sends $P\sqcup Q$ to a labeled partition $P'\in\PP([p_1+q_1],[p_2+q_2])$. There then exist permutations $\sigma_1\in\Sigma_{p_1+q_1}$ and $\sigma_2\in\Sigma_{p_2+q_2}$ such that $P'':=(\sigma_1\times\sigma_2)\cdot P'$ is a standard labeled partition.
    
    Let $\sigma\in\Sigma_{l+r}$ be a permutation such that $\sigma(\lambda_1,\ldots,\lambda_l,\lambda_1',\ldots,\lambda_r')$ is a partition (i.e.\ weakly decreasing). Let us write $\lambda''$ for the resulting partition and $k'':=\sigma(k_1,\ldots,k_l,k_1',\ldots,k_r')$. We then have $P''=P_{\lambda'',k''}$, so it follows that
    \begin{align*}
        \kappa_{S\sqcup T}(P\sqcup P')&=((\sigma_1\circ(f_1\sqcup\bar{g}_1))^{-1}\times(\sigma_2\circ(f_2\sqcup\bar{g}_2))^{-1})_*(\kappa(\lambda'',k'')\otimes v_{p_1+q_1}\otimes v_{p_2+q_2})\\
        &=((\sigma_1\circ(f_1\sqcup\bar{g}_1))^{-1}\times(\sigma_2\circ(f_2\sqcup\bar{g}_2))^{-1})_*(h_{\lambda_1'',k''_1}\cdots h_{\lambda_{l+r}'',k_{l+r}''}\otimes v_{p_1+q_1}\otimes v_{p_2+q_2})\\
        &=((f_1\sqcup\bar{g}_1)^{-1}\times(f_2\sqcup\bar{g}_2)^{-1})_*\big(\kappa(\lambda_P,k_P)\kappa(\lambda_{P'},k_{P'})\otimes v_{p_1+q_1}\otimes v_{p_2+q_2}\big)\\
        &=(f_1^{-1}\times f_2^{-1})_*(\kappa(\lambda_P,k_P)\otimes v_{p_1}\otimes v_{p_2})(g_1^{-1}\times g_2^{-1})_*(\kappa(\lambda_{P'},k_{P'})\otimes v_{q_1}\otimes v_{q_2})\\
        &=\kappa_S(P)\kappa_T(P').
    \end{align*}
    The fact that there is no sign change in the third step is explained by a similar argument as in the proof of Proposition \ref{prop:kappa-welldefined}, as $\sigma_1^{-1}\times\sigma_2^{-1}$ acts by permuting the classes $h_{\lambda''_i,k''_i}$, the sign change of which is cancelled by the action on the sign representations.
\end{proof}

\section{Proof of Theorem B}\label{sec:catdescriptionofcohomology} 

Having described the functor $H^*(\Aut(F_n),K)\cong\PP\otimes\det:\wBr_n\to\Gr(\bQ\dashmod)$, we want to relate this functor to the functor $[H^*(\IA_n,\bQ)\otimes K]^{\GL_n(\bZ)}:\wBr_n\to\Gr(\bQ\dashmod)$, in order to apply Proposition \ref{prop:wBr1}. Similarly to in \cite{KupersRandal-Williams}, this will be done using the Hochschild-Serre spectral sequence associated to the short exact sequence defining $\IA_n$. However, this spectral sequence has already been analysed thoroughly in \cite{HabiroKatada2023stable}, so we will recall their work and expand some of their results slightly, after which we reframe the results in the categorical perspective.

\subsection{The $\GL_n$-invariants of $H^*(\IA_n,\bQ)$}\label{subsec:invariantstrivialcoeffs}

In \cite[Section 6]{HabiroKatada2023stable}, Habiro and Katada define classes which we will denote by $y_{4k}\in H^{4k}(\IA_n,\bQ)^{\GL_n(\bZ)}$ for all $k\ge 1$, using what they call ``anti-transgression maps''. Let us start by recalling the general definition of these maps.

Suppose that we have a first quadrant cohomological spactral sequence $E$ such that for some $N>0$, $E_\infty^{p,q}=0$ whenever $0<p+q<N$. Thus it follows that for $0\le k<N$ we have $E_{k+2}^{0,k}=E_{k+2}^{k+1,0}=0$, since there can be no non-zero differentials to or from these terms on the $E_{k+2}$-page. Thus the differential
$$d_{k+1}:E_{k+1}^{0,k}\to E_{k+1}^{k+1,0}$$
is an isomorphism, so we can define a linear map
$$\varphi_k:=\iota_k\circ (d_{k+1})^{-1}\circ\pi_{k+1}:E_{2}^{k+1,0}\to E_{2}^{0,k}$$
where $\pi_{k+1}:E_2^{k+1,0}\twoheadrightarrow
 E_{k+1}^{k+1,0}$ and $\iota_k:E_{k+1}^{0,k}\hookrightarrow E_2^{0,k}$ are the edge homomorphisms. We call $\varphi_k$ an \textit{anti-transgression} map.

 We will apply this notion to the Hochschild-Serre spectral sequence associated to the short exact sequence
     \begin{equation}
         1\to\IA_n\to\Aut(F_n)\to\GL_n(\bZ)\to 1.
     \end{equation}
     Recall that for $n\gg *>0$, we have $H^*(\Aut(F_n),\bQ)=0$, so for $n\gg k$, we get a well-defined anti-transgression map
     $$\varphi_k:H^{k+1}(\GL_n(\bZ),\bQ)\to H^k(\IA_n)^{\GL_n(\bZ)}.$$
     Since $H^*(\GL_n(\bZ),\bQ)\cong\exterior{*}\{x_{4i+1}\mid i\ge 1\}$ in a range $n\gg *$, where the generators are indexed by their degrees, we can define elements
     $$y_{4i}:=\varphi_{4i}(x_{4i+1})\in H^{4i}(\IA_n,\bQ)^{\GL_n(\bZ)},$$
     for $n\gg 4i+1$, giving us a well defined algebra map
     $$\varphi:\bQ[y_4,y_8,\ldots]\to H^*(\IA_n,\bQ)^{\GL_n(\bZ)},$$
     where we define the map to be zero on the generators in degrees where the anti-transgression map is not defined. The goal will be to prove the following theorem:

\begin{theorem}[{cf. \cite[Theorem 6.3] {HabiroKatada2023stable}}]\label{thm:GL-invariants}
    Suppose that there is a $Q\ge 0$ such that $H^*(\IA_n,\bQ)$ satisfies Borel vanishing in degrees $*<Q $. Then there is a range $n\gg *$ where the algebra homomorphism
    $$\varphi:\bQ[y_4,y_8,\ldots]\to H^*(\IA_n,\bQ)^{\GL_n(\bZ)},$$
    is an isomorphism in degrees $*\le Q$.
\end{theorem}

 To do this, we will prove the following proposition, which is a very slight generalization of \cite[Theorem 6.4]{HabiroKatada2023stable}, improving the bound on the isomorphism of graded algebras by 1.

\begin{proposition}\label{prop:invariants-SSeq}
    Let $P,Q,N\in\{1,2,\ldots,\infty\}$ and $E$ be a graded commutative, multiplicative, first quadrant, cohomological spectral satisfying the following:\begin{enumerate}
        \item For $p\le P$ and $q\le Q$, the multiplication map $E_2^{p,0}\otimes E_2^{0,q}\to E_2^{p,q}$ is an isomorphism.
        \item  For $p+q\le N$, we have $E_\infty^{p,q}=0$.
        \item  There is a graded vector space $V=\bigoplus_{k\ge 2} V^k$ such that
        $$E_2^{*,0}\cong S(V),$$
        in degrees $*\le P$, where $S(V)$ denotes the free graded commutative algebra on $V$.
    \end{enumerate}
    Then the restriction of the the antitransgression map $\varphi_k:E_2^{k+1,0}\to E_2^{0,k}$ to $V^{k+1}$ is injective for all $1\le k\le 1+ \mathrm{min}(P-2,Q,N-2)$ and 
    $$E_2^{0,*}\cong S(V[-1])$$
    as graded algebras in degrees $*\le 1+\mathrm{min}(P-2,Q,N-2)$.
\end{proposition}

The proof of this proposition is based on that of \cite[Theorem 6.4]{HabiroKatada2023stable} and is mostly the same, but with some modification to get the improved bound.

We will start by proving a series of lemmas. The first is surely a standard result in homological algebra, but we include an elementary proof for completeness:

\begin{lemma}\label{lemma:chainmap}
    Suppose $f:A\to B$ is a map of cochain complexes such that for some $i$, $f^i:A^i\to V^i$ is surjective and $f^{i+1}:A^{i+1}\to B^{i+1}$ is injective. Then $H(f)^i$ is surjective and $H(f)^{i+1}$ is injective.
\end{lemma}

\begin{proof}
    First, we introduce the notation $d_A^j:A^j\to A^{j+1}$ and $d_B^j:B^j\to B^{j+1}$ for the differentials in the chain complexes, for all $j$. We note that by the definition of a chain map, we have \begin{align}
        f^j(\ker(d_A^j))\subseteq \ker(d_B^j)\label{eq:row1},\\
        f^j(\mathrm{im}(d_A^j))\subseteq \mathrm{im}(d_B^{j})\label{eq:row2},
    \end{align}
for all $j$.

    Let us start by proving the surjectivity of $H(f)^i$. Let $w\in\ker(d_B^i)\subset B^i$ and $v\in A^i$ such that $f^i(v)=w$. Since $0=d_B^i(w)=d_B^i(f^i(v))=f^{i+1}(d_A^i(v))$ and $f^{i+1}$ is injective, we have $v\in\ker(d_A^i)$. Thus the restriction of $f^i$ to $\ker(f^i)$ is a surjection onto $\ker(d_B^i)$. By (\ref{eq:row2}) with $j=i$, it follows that $H(f)^i$ is surjective.

    Next, let us prove injectivity of $H(f)^{i+1}$. By (\ref{eq:row1}) for $j=i+1$ and the assumption, we have that the restriction of $f_2$ to $\ker(d_A^{i+1})$ is an injection into $\ker(d_B^{i+1})$. Suppose $v\in A^{i+1}$ is such that $f^{i+1}(v)\in\mathrm{im}(d_B^i)$, i.e.\ $f^{i+1}(v)=d_B^i(w)$ for some $w\in B^i$. By surjectivity of $f^i$, we have some $v'\in A^i$ such that $f^i(v')=w$, so $f^{i+1}(v)=d_B^i(f^i(v'))=f^{i+1}(d_A^i(v'))$. Since $f^{i+1}$ is injective, we have $v=d_A^i(v')$, so $v\in\mathrm{im}(d_A^i)$. Thus $H(f)^{i+1}$ is injective.
\end{proof}

\begin{lemma}\label{lemma:cokernel}
    Let $f:{^1}E\to {^2}E$ be a morphism of first quadrant cohomological spectral sequences, where ${^1}E$ is such that for some $k$, ${^1}E_{k+1}^{k+1,0}=0$. If $f_2^{p,k-p}$ is surjective for all $1\le p\le k-1$ and $f_2^{p,k+1-p}$ is injective for all $2\le p$, then
    $$0\to {^1}E_{2}^{k+1,0}\overset{f_2^{k+1,0}}{\longrightarrow}{^2}E_2^{k+1,0}\twoheadrightarrow {^2}E_{k+1}^{k+1,0}\to 0$$
    is a short exact sequence, where the right map is the edge homomorphism.
\end{lemma}

\begin{proof}
    First, we will inductively prove that for each $i\ge 2$, the map $f_i^{p,k-p}$ is surjective for all $1\le p\le k-1$ and $f_i^{p,k+1-p}$ is injective for all $i\le p$. For $i=2$, this is the assumption of the lemma, so suppose the statement holds for some $i\ge 2$. This means that for $1\le p\le k-1$ we have a chain map
    \[\begin{tikzcd}
        \cdots\arrow[r]&{^1}E_{i}^{p,k-p}\arrow[r,"d_i^1"]\arrow[d,"f_i^{p,k-p}", two heads]& {^1}E_i^{p+i,k+1-(p+i)}\arrow[r]\arrow[d,"f_i^{p+i,k+1-(p+i)}", hook]&\cdots\\
        \cdots\arrow[r]&{^2}E_i^{p,k-p}\arrow[r,"d_i^2"]&{^2}E_i^{p+i,k+1-(p+i)}\arrow[r]&\cdots
    \end{tikzcd}\]
    Thus it follows from Lemma \ref{lemma:chainmap} that $f_{i+1}^{p,k-p}$ is surjective for $1\le p\le k-1$ and that $f_{i+1}^{p,k+1-(p)}$ is injective for $i\le p$ (note that injectivity for $p>k+1$ is trivial since these are first quadrant spectral sequences).
    
    Now let us use this to prove the statement of the lemma. By the first part of the proof, for each $2\le i\le k$ we have a commutative diagram
    \[\begin{tikzcd}
        \cdots\arrow[r]&{^1}E_i^{k+1-i,i-1}\arrow[r,"d_i^1"]\arrow[d,"f_i^{k+1-i,i-1}",two heads]&{^1}E_i^{k+1,0}\arrow[d,"f_i^{k+1,0}",hook]\arrow[r]&0\\
        \cdots\arrow[r]&{^2}E_i^{k+1-i,i-1}\arrow[r,"d_i^2"]& {^2}E_i^{k+1,0}\arrow[r]&0
    \end{tikzcd}\]
    It follows that $f_i^{k+1,0}(\mathrm{im}(d_i^1))=\mathrm{im}(d_i^2)$, so $\mathrm{coker}(f_{i+1}^{k+1,0})\cong\mathrm{coker}(f_{i}^{k+1,0})$. However, by assumption ${^1}E_{k+1}^{k+1,0}=0$, so ${^2}E_{k+1}^{k+1,0}\cong\mathrm{coker}(f_{k+1}^{k+1,0})\cong\mathrm{coker}(f_2^{k+1,0})$, which finishes the proof.
\end{proof}

\begin{lemma}\label{lemma:injectivity-SSeq}
    Let $f:{^1}E\to {^2}E$ be a morphism of first quadrant cohomological spectral sequences where there is a $k\ge 1$ such that ${^1}E_{k+2}^{0,k}=0$. If $f_2^{p,k-p}$ is surjective for all $p\le k-1$ and $f_2^{p,k+1-p}$ is injective for all $p\ge 2$, then $f_2^{0,k}$ is injective. 
    
\end{lemma}

\begin{proof}
    We start by inductively proving that for $i\ge 2$, $f_i^{p,k-p}$ is surjective for all $p\le k-1$ and $f_i^{p,k+1-p}$ is injective for all $p\ge 2$. For $i=2$ this is the assumption of the lemma, so let us assume it is true for some $i\ge 2$. For any $2-i\le p\le k-1$, we then have a chain map
     \[\begin{tikzcd}
        \cdots\arrow[r]&{^1}E_{i}^{p,k-p}\arrow[r,"d_i^1"]\arrow[d,"f_i^{p,k-p}", two heads]& {^1}E_i^{p+i,k+1-(p+i)}\arrow[r]\arrow[d,"f_i^{p+i,k+1-(p+i)}", hook]&\cdots\\
        \cdots\arrow[r]&{^2}E_i^{p,k-p}\arrow[r,"d_i^2"]&{^2}E_i^{p+i,k+1-(p+i)}\arrow[r]&\cdots
    \end{tikzcd}\]
    so it follows by Lemma \ref{lemma:chainmap} that for $2-i\le p\le k-1$, $f_{i+1}^{p,k-p}$ is surjective and $f_{i+1}^{p+i,k+1-(p+i)}$ is injective. However, for $p\le 2-i\le 0$, $f_{i+1}^{p,k-p}$ is trivially surjective. If $p>k-1$, then $p+i>k-1+i\ge k+1$, so $k+1-(p-i)<0$, which means that $f_{i+1}^{p+i,k+1-(p+i)}$ is injective for trivial reasons as well, and thus $f_{i+1}^{p,k+1-p}$ is injective for all $p\ge 2$.

    Now let us use this to prove the statement of the lemma. We have filtrations
    \begin{align*}
        0= {^1}E_{k+2}^{0,k}\subseteq{^1}E_{k+1}^{0,k}\subseteq\cdots\subseteq {^1}E_{3}^{0,k}\subseteq {^1}E_2^{0,k}\\
        \cdots\subseteq {^2}E_{k+2}^{0,k}\subseteq{^2}E_{k+1}^{0,k}\subseteq\cdots\subseteq {^2}E_{3}^{0,k}\subseteq {^2}E_2^{0,k}
    \end{align*}
    and the map $f_2^{0,k}$ is compatible with these filtrations, since $f_i^{0,k}|_{{^1}E_{i+1}^{0,k}}=f_{i+1}^{0,k}$. Inductively, we will prove that for each $i\ge 2$, $f_i^{0,k}$ is injective. For $i\ge k+2$, this is trivially satisfied, so assume it is true for some $2<i\le k+2$. By the first part, we have a commutative diagram
    \[\begin{tikzcd}
        {^1}E_{i-1}^{0,k}\arrow[r,"d_{i-1}^1"]\arrow[d, "f_{i-1}^{0,k}"]&{^1}E_{i-1}^{i-1,k+1-(i-1)}\arrow[d,"f_{i-1}^{i-1,k+1-(i-1)}",hook]\\
        {^2}E_{i-1}^{0,k}\arrow[r,"d_{i-1}^2"]&{^2}E_{i-1}^{i-1,k+1-(i-1)}
    \end{tikzcd}\]
    By the commutative diagram we have $f_{i-1}^{0,k}(\ker(d_{i-1}^1))\subseteq \ker(d_{i-1}^2)$, so we get an induced commutative diagram
    \[\begin{tikzcd}
        {^1}E_{i-1}^{0,k}/\ker(d_{i-1}^1)\arrow[r,"\overline{d}_{i-1}^1",hook]\arrow[d, "\overline{f}_{i-1}^{0,k}"]&{^1}E_{i-1}^{i-1,k+1-(i-1)}\arrow[d,"f_{i-1}^{i-1,k+1-(i-1)}",hook]\\
        {^2}E_{i-1}^{0,k}/\ker(d_{i-1}^2)\arrow[r,"\overline{d}_{i-1}^2", hook]&{^2}E_{i-1}^{i-1,k+1-(i-1)}
    \end{tikzcd}\]
    where the overlined maps are those induced on the quotients. It follows that $\overline{f}_{i-1}^{0,k}$ is injective. By the inductive assumption, the restriction of $f_{i-1}^{0,k}$ to the kernel of $d_{i-1}^1$ is also injective, so it follows that $f_{i-1}^{0,k}$ is injective as well. 
\end{proof}

We are now equipped to prove the proposition.

\begin{proof}[Proof of Proposition \ref{prop:invariants-SSeq}]

    The beginning of the proof is just the same as the proof of \cite[Theorem 6.4]{HabiroKatada2023stable}, but let us recall it for completeness, and so that the end of the proof is easier to follow. The propositions only differs in the case where at least one of $P,Q$ and $N$ is finite, so this is the case we will prove. For this purpose, let $K=2+\mathrm{min}(P-2,Q,N-2)$. 
    
    For $2\le k\le K$, we define a graded commutative, first quadrant, multiplicative spectral sequence ${^k}{E}$ by
    $$E_2^{p,q}:=S(V^k)^p\otimes S(V[-1]^{k-1})^q.$$
    with the only non-zero differential being on the $E_k$-page, where $d_k^{0,k-1}=\id_{V^k}:{^k}{E}_k^{0,k-1}\to {^k}{E}_k^{k,0}$. The other non-zero differentials are determined by the Leibniz rule and for each $i\ge 1$ we have an exact sequence
    $$0\to S^i(V[-1]^{k-1})\to S^1(V^k)\otimes S^{i-1}(V[-1]^{k-1})\to\cdots\to S^i(V^k)\to 0.$$
    For $p+q>0$ and all $j>k$ we thus get ${^k}{E}_j^{p,q}=0$. For each $2\le k\le K$ we now define a new spectral sequence 
    $${^{\le k}}{E}:=\bigotimes_{2\le i\le k} {^i}{E}.$$

    Now the proof diverges slightly from that of \cite[Theorem 6.4]{HabiroKatada2023stable}, although it is still mostly the same. By induction we will prove that for each $2\le k\le K$, the restriction of $\varphi_{k-1}:E_2^{k,0}\to E_2^{0,k-1}$ to $V_k$ is injective and that there are morphisms
    $${^k}f:{^{\le k}}{E}\to E$$
    such that ${^k}f_2^{p,0}$ is an isomorphism for $0\le p\le k$, ${^k}f_2^{0,q}$ is an isomorphism for $0\le q\le k-1$ and such that ${^{k+1}}f|_{^{\le k}E}={^k} f$ for each $2\le k\le K-1$.

    For $k=2$, we have that $d_2^{0,1}:E_2^{0,1}\to E_2^{2,0}$ is an isomorphism by the assumption (2), and we have $E_2^{2,0}\cong V^2$, by assumption (3), so $\varphi_1$ is actually an isomorphism and in particular injective. Thus there is a morphism ${^2}f:{^{\le 2}}E={^2}E\to E$ determined entirely by ${^2}f_2^{2,0}=\id_{V^2}$ and ${^2}f_2^{0,1}=\id_{V[-1]^1}$, both of which are isomorphisms.

    Now let us assume the statement holds for some $2\le k<K$. By this assumption and condition (1), it follows that ${^k}f_{2}^{p,q}$ is an isomorphism for $0\le p\le k$ and $0\le q\le k-1$. In particular, ${^k}f_2^{p,k-p}$ is an isomorphism for all $1\le p\le k-1$ and ${^k}f_2^{p,k+1-p}$ is an isomorphism for all $2\le p\le k$, and injective for $p\ge k+1$ (since ${^k}f_2^{*,0}$ is injective in all degrees). Since the cokernel of the injective map ${^k}f_2^{k+1,0}:{^{\le k}}E_2^{k+1,0}\to E_2^{k+1,0}$ is precisely $V^{k+1}$ it follows by Lemma \ref{lemma:cokernel} that the restriction of the edge map $E_2^{k+1,0}\to E_{k+1}^{k+1,0}$ to $V^{k+1}$ is an isomorphism. Thus the restriction of the anti-transgression map $\varphi_k:E_2^{k+1,0}\to E_2^{0,k}$ is injective. 
    
    In the same way as in the base case, we may thus extend ${^k}f$ to a map ${^{k+1}}f:{^{k+1}}E\to E$ such that ${^{k+1}}f_{2}^{p,0}$ is an isomorphism for $0\le p\le k+1$. Since ${^{k+1}}f_{\infty}^{p,q}$ is an isomorphism for all $p+q\le N$, by condition (2), it follows by Zeeman's comparison theorem \cite[Theorem 2]{ZeemanComparison} that ${^{k+1}}f_2^{0,q}$ is an isomorphism for all $q\le k-1$ and surjective for $q=k$. By condition (1), we once again have that ${^{k+1}}f_2^{p,q}$ is an isomorphism for $p\le k+1$ and $q\le k-1$. In particular, ${^{k+1}}f_2^{p,k+1-p}$ is injective for all $p\ge 2$ and since ${^{k+1}}f_2^{0,k}$ is surjective, we get in particular that ${^{k+1}}f_2^{p,k-p}$ is surjective for all $p\le k$. Thus it follows by Lemma \ref{lemma:injectivity-SSeq} that ${^{k+1}}f_2^{0,k}$ is also injective and thus an isomorphism. This finishes the induction and for $k=K-1$, we get the claim of the proposition, for $K$ finite.     
\end{proof}

Let us now use this to prove Theorem \ref{thm:GL-invariants}:

\begin{proof}[Proof of Theorem \ref{thm:GL-invariants}]
    We need to verify that the Hochschild-Serre spectral sequence associated to the short exact sequence
    $$1\to\IA_n\to\Aut(F_n)\to\GL_n(\bZ)\to 1$$
    satisfies the hypotheses of Proposition \ref{prop:invariants-SSeq}. We already saw in the paragraph preceeding the theorem that (2) and (3) are satisfied. We have $E_2^{p,0}\cong H^p(\GL_n(\bZ),\bQ)$, whereas $E_2^{0,q}\cong  H^q(\IA_n,\bQ)^{\GL_n(\bZ)}$, so (1) is precisely Borel vanishing, which we have assumed to hold. 
\end{proof}

\subsection{The $\GL_n(\bZ)$-invariants of $H^*(\IA_n,\bQ)\otimes K(S)$}

In order to apply Proposition \ref{prop:wBr1} to compute stable cohomology groups of $\IA_n$, we need to compute the values of the functor $[H^*(\IA_n,\bQ)\otimes K]^{\GL_n(\bZ)}:\wBr_n\to\bQ\text{-}\mathrm{mod}$. For any pair of finite sets $S=(S_1,S_2)$, the inclusion $\IA_n\hookrightarrow \Aut(F_n)$ induces a map
$$H^*(\Aut(F_n),K(S))\to H^*(\IA_n,K(S))\cong H^*(\IA_n,\bQ)\otimes K(S),$$
since $\IA_n$ by definition is the kernel of the action on $K(S)$. Furthermore, since the $\GL_n(\bZ)$-action on the cohomology of $\IA_n$ is induced by the conjugation action of $\Aut(F_n)$ on itself, this is a $\GL_n(\bZ)$-equivariant map, with $H^*(\Aut(F_n),K(S))$ considered as a trivial reprsentation. Thus it follows that the image lands in the $\GL_n(\bZ)$-invariants of $H^*(\IA_n,\bQ)\otimes K(S)$. Let us denote this map
$$\iota_S:H^*(\Aut(F_n),K(S))\to [H^*(\IA_n,\bQ)\otimes K(S)]^{\GL_n(\bZ)}.$$
Tensoring with $H^*(\IA_n,\bQ)^{\GL_n(\bZ)}$ and taking the cup product we get a composite map
$$\xi_S:H^*(\Aut(F_n),K(S))\otimes H^*(\IA_n,\bQ)^{\GL_n(\bZ)}\to [H^*(\IA_n,\bQ)\otimes K(S)]^{\GL_n(\bZ)},$$
The following is a consequence of \cite[Lemma 7.7]{HabiroKatada2023stable}:

\begin{theorem}
    Suppose $Q\ge 0$ is such that $H^*(\IA_n,\bQ)$ satisifies Borel vanishing in degree $*< Q$. For any pair of finite sets $S=(S_1,S_2)$ there is then a stable range $n\gg *$ in which the map $\xi_S$ is an isomorphism in all degrees $\le Q$.
\end{theorem}

\begin{proof}
    The proof is almost identical to that of \cite[Lemma 7.7]{HabiroKatada2023stable}, by noting that $K(S)$ decomposes into irreducible representations $V_{\lambda,\mu}$ with $|\lambda|=|S_1|$ and $|\mu|=|S_2|$ (denoted $V_{\overline{\lambda}}$ by Habiro and Katada). Keeping track of degrees in that proof, one notes that it is sufficient for Borel vanishing to be satisfied in degrees $*<Q$ to conclude an isomorphism in degrees $\le Q$.
\end{proof}

Combining this with Theorem \ref{thm:twistedcoeffs} and \ref{thm:GL-invariants}, we can conclude the following:

\begin{corollary}\label{corollary:invariants}
    For any pair of finite sets $S=(S_1,S_2)$, there is a linear map
    $$\Phi_S^t:\PP(S)\otimes\det(S)\otimes\bQ[y_4,y_8,\ldots]\to [H^*(\IA_n,\bQ)\otimes K(S)]^{\GL_n(\bZ)}.$$
    If $Q\ge 0$ is such that $H^*(\IA_n,\bQ)$ satisfies Borel vanishing in degrees $*<Q$, there is a range $n\gg *$ such that $\Phi_S^t$ is an isomorphism in degrees $\le Q$.
\end{corollary}

Let us now apply the machinery of Section \ref{sec:walledBrauer} to prove Theorem \ref{theoremB}.

\subsection{Proof of Theorem \ref{theoremB}} 

Let us start by recalling the categorical version of the main theorem:

\theoremB*

\begin{proof}
We start by noting that by adjunction, for any pair of finite sets $S=(S_1,S_2)$, the map
$$\Phi_S^t:i_*(\PP')(S)\otimes\det(S)\otimes\bQ[y_4,y_8,\ldots]\to [H^*(\IA_n,\bQ)\otimes K(S)]^{\GL_n(\bZ)}$$
defines a $\GL_n(\bZ)$-equivariant map
$$\Psi_S^t:K^\vee(S)\otimes \PP'(S)\otimes\det(S)\otimes\bQ[y_4,y_8,\ldots]\to H^*(\IA_n,\bQ).$$
The maps $\Phi_S^t$ define the components of a natural transformation $\Phi^t$ of functors $\wBr_n\to\Gr(\bQ\dashmod)$, so $\Psi_S^t$ define the components of a natural transformation $\Psi^t$ of functors $\wBr_n\to\Gr(\Rep(\GL_n(\bZ)))$. This means that they define a graded ring homomorphism
$$i^*(K^\vee)\otimes^{\dwBr}(\PP'\otimes\det)\otimes\bQ[y_4,y_8,\ldots]\to H^*(\IA_n,\bQ)$$
By Corollary \ref{corollary:invariants} above, the natural transformation $\Phi^t$ is a natural isomorphism is a range of degrees, so by Proposition \ref{prop:wBr1}, this is an isomorphism of graded rings onto $H^*(\IA_n,\bQ)^\alg$ in the same range of degrees.
\end{proof}

\section{Ring structure}\label{sec:ring-structure}

In this section, we prove Theorem $\ref{theoremA}$ by describing the ring
$$R:=i^*(K^\vee)\otimes^{\dwBr}(\PP'\otimes\det)$$
in more explicit terms, following the example of \cite[Section 5]{KupersRandal-Williams}. Similarly to there, we note that by the universal property of coends, there is for any $S=(S_1,S_2)\in\dwBr$ a $\Sigma_{S_1}\times\Sigma_{S_2}\times\GL_(\bZ)$-equivariant map
$$K^\vee(S)\otimes \PP'(S)\otimes\det(S)\to R.$$
In particular, if $S=([k],[1])$, the trivial labeled partition $\{[k],1\}$ defines a map
$$H^\vee(n)^{\otimes k}\otimes H(n)\otimes\det(\bQ^k)\to R.$$
Similarly, if $S=([k],\varnothing)$, we get a map
$$H^\vee(n)^{\otimes k}\otimes\det(\bQ^k)\to R.$$
Forgetting the structure as $\Sigma_k$-representations, we thus get linear maps
\begin{align*}
    \kappa_1:H^\vee(n)^{\otimes k}\otimes H(n)&\to R\\
    \kappa_0:H^{\vee}(n)^{\otimes k}&\to R,
\end{align*}
where the $\Sigma_k$-equivariance can be recorded by the relations\begin{align}\label{eq:symmetryrelations}
    \kappa_1(f_{\sigma(1)}\otimes f_{\sigma(2)}\otimes\cdots\otimes f_{\sigma(k)}\otimes v)&=\sgn(\sigma)\kappa_1(f_1\otimes f_2\otimes \cdots \otimes f_k\otimes v),\\
    \kappa_0(f_{\sigma(1)}\otimes f_{\sigma(2)}\otimes\cdots\otimes f_{\sigma(k)})&=\sgn(\sigma)\kappa_0(f_1\otimes f_2\otimes \cdots \otimes f_k).
\end{align}

Let $S=(\{s_1,s_2,\ldots, s_p\},\varnothing)$ and $T=(\{t_1,t_2,\ldots,t_q\},\varnothing)$. Furthermore, let $S':=(S_1,\{s\})$ and $T':=(\{t\}\sqcup T_1,\varnothing)$. For $\bar{f}\in H^\vee(n)^{\otimes S_1}$ and $\bar{g}\in H^\vee(n)^{\otimes T_1}$. We then get that the element
$$\sum_{i=1}^n (\bar{f}\otimes e_i\otimes e^\#_i\otimes \bar{g})\otimes\{S',T'\}\otimes(s_1\wedge\cdots\wedge s_p\wedge t\wedge t_1\wedge\cdots\wedge t_q\otimes s),$$
in $K^\vee_{S'\sqcup T'}(n)\otimes \PP'(S'\sqcup T')\otimes\det(S'\sqcup T')$, and the element
$$(\bar{f}\otimes\bar{g})\otimes\{S,T\}\otimes(s_1\wedge\cdots\wedge s_p\wedge t_1\wedge\cdots\wedge t_q),$$
in $K^\vee_{S\sqcup t}(n)\otimes\PP'(S\sqcup T)\otimes\det(S\sqcup T)$, are identified in the coend $R$. This thus gives us the relation
\begin{align}\label{eq:contraction-rel1}
    \sum_{i=1}^n\kappa_1(\bar{f}\otimes e_i)\otimes\kappa_0( e^\#_i\otimes \bar{g})=\kappa_0(\bar{f}\otimes \bar{g}) 
\end{align}
in $R$. We can also note that the same calculation holds if we replace $\bar{g}$ by $\bar{g}\otimes v\in H^{\vee}(n)^{\otimes T_1}\otimes H(n)$, so we also have the relation
\begin{align}\label{eq:contraction-rel2}
    \sum_{i=1}^n\kappa_1(\bar{f}\otimes e_i)\otimes\kappa_1( e^\#_i\otimes\bar{g}\otimes v)=\kappa_1(\bar{f}\otimes\bar{g}\otimes v).
\end{align}
Similarly, if we let $S$ be as above and $T=(\{t_1\},\{t_2\})$ we have that the element
$$\sum_{i=1}^n(\bar{f}\otimes  e^\#_i\otimes e_i)\otimes\{(S_1\sqcup\{t_1\},t_2)\}\otimes(s_1\wedge \cdots\wedge s_p\wedge t_1)\otimes (t_2)$$
in $K_{S\sqcup T}^\vee(n)\otimes\PP'(S\sqcup T)\otimes\det(S\sqcup T)$ and the element
$$\bar{f}\otimes \{S\}\otimes(s_1\wedge\cdots\wedge s_p)$$
in $K_S^\vee(n)\otimes\PP'(S)\otimes\det(S)$ are identified in the coend $R$, which gives us the relation
\begin{align}\label{eq:contraction-rel3}
    \kappa_0(f)=\sum_{i=1}^n\kappa_1(f\otimes e^\#_i\otimes e_i).
\end{align}
By iteratively applying the relation (\ref{eq:contraction-rel2}), we get that
\begin{align*}
    \kappa_1(f_1\otimes f_2\otimes\cdots\otimes f_k\otimes v) &= \sum_{i=1}^n \kappa_1(f_1\otimes f_2\otimes\cdots\otimes f_{k-1}\otimes e_i)\otimes\kappa_1( e^\#_i\otimes f_k\otimes v)\\
    &=\cdots\\
    &=\sum_{i_1,\ldots,i_{k-2}}\kappa_1(f_1\otimes f_2\otimes e_{i_1})\otimes\cdots\otimes\kappa_1( e^\#_{i_{k-2}}\otimes f_k\otimes v)
\end{align*}
and similarly using relation (\ref{eq:contraction-rel1}) and in the last step relation (\ref{eq:contraction-rel3}), we get
\begin{align*}
    \kappa_0(f_1\otimes\cdots\otimes f_k)&=\sum_{i_1,\ldots,i_{k-2}}\kappa_1(f_1\otimes f_2\otimes e_{i_1})\otimes\cdots\otimes\kappa_1( e^\#_{i_{k-3}}\otimes f_{k-1}\otimes e_{i_{k-2}})\otimes \kappa_0( e^\#_{i_{k-2}}\otimes f_{k})\\
    &=\sum_{i_1,\ldots,i_{k-1}}\kappa_1(f_1\otimes f_2\otimes e_{i_1})\otimes\cdots\otimes \kappa_1( e^\#_{i_{k-2}}\otimes f_{k}\otimes e_{i_{k-1}})\otimes \kappa_0( e^\#_{i_{k-1}})\\
    &=\sum_{i_1,\ldots,i_{k}}\kappa_1(f_1\otimes f_2\otimes e_{i_1})\otimes\cdots\otimes \kappa_1( e^\#_{i_{k-2}}\otimes f_{k}\otimes e_{i_{k-1}})\otimes \kappa_1( e^\#_{i_{k-1}}\otimes e^\#_{i_k}\otimes e_{i_k}).
\end{align*}
We thus see that the subring of $R$ generated by all classes of the form $\kappa_1(f_1\otimes\cdots\otimes f_k\otimes v)$ and $\kappa_0(f_1\otimes\cdots\otimes f_k)$ is in fact generated by the classes of the form $\kappa_1(f_1\otimes f_2\otimes v)$. We see that applying the relation (\ref{eq:contraction-rel2}), then (\ref{eq:symmetryrelations}) and finally (\ref{eq:contraction-rel2}) again, gives us a relation
\begin{align*}
    \sum_{i=1}^n\kappa_1(f_1\otimes f_2\otimes e_i)\otimes\kappa_1( e^\#_i\otimes f_3\otimes v)&=\kappa_1(f_1\otimes f_2\otimes f_3\otimes v)\\
    &=(-1)^2\kappa_1(f_3\otimes f_1\otimes f_2\otimes v) \\
    &=\sum_{i=1}^n\kappa_1(f_3\otimes f_1\otimes e_i)\otimes\kappa_1( e^\#_i\otimes f_2\otimes v).
\end{align*}

We use this to define the following ring:

\begin{definition}\label{def:ringstructure}
    Let $R_{\pres}$ denote the graded commutative ring generated by the classes $\kappa_1(f_1\otimes f_2\otimes v)$, with the relations given by\begin{enumerate}[(1)]
        \item linearity in $f_1$, $f_2$ and $v$,
        \item\label{eq:rel2-Rpres} $\kappa_1(f_1\otimes f_2\otimes v)=-\kappa_1(f_2\otimes f_1\otimes v)$,
        \item\label{eq:rel3-Rpres} $\sum_{i=1}^n\kappa_1(f_1\otimes f_2\otimes e_i)\otimes\kappa_1( e^\#_i\otimes f_3\otimes v)=\sum_{i=1}^n\kappa_1(f_3\otimes f_1\otimes e_i)\otimes\kappa_1( e^\#_i\otimes f_2\otimes v)$.
    \end{enumerate}
\end{definition}

\begin{remark}
    Since $\kappa_1(f_1\otimes f_2\otimes v)$ has degree 1, we have in other words that $R_{\pres}$ is isomorphic to the quotient of the graded ring $\exterior{*}(\exterior{*}H^\vee(n)\otimes H(n))$ by the relation 
    $$\sum_{i=1}^n (f_1\wedge f_2\otimes e_i)\wedge( e^\#_i\wedge f_3\otimes v)=\sum_{i=1}^n(f_3\wedge f_1\otimes e_i)\wedge ( e^\#_i\wedge f_2\otimes v).$$
\end{remark}

Since the relations in $R_{\pres}$ are satisfied by the classes $\kappa_1(f_1\otimes f_2\otimes v)\in R$, we have a ring homomorphism $R_{\pres}\to R$. Theorem \ref{theoremA} is then a consequence of the following proposition:

\begin{proposition}\label{prop:ringstructure}
   The ring homomorphism $R_{\pres}\to R$ is surjective. In a stable range of degrees, it is an isomorphism.
\end{proposition}

To prove that proposition, we once again follow the example of \cite{KupersRandal-Williams} and note that since $R$ and $R_{\pres}$ both are graded representations of $\GL_n(\bZ)$ which are finite dimensional in each degree, it suffices to show that for any pair of finite sets $S=(S_1,S_2)$, the induced map
$$[R_{\pres}\otimes K_S(n)]^{\GL_n(\bZ)}\to[R\otimes K_S(n)]^{\GL_n(\bZ)}$$
is surjective and an isomorphism in a stable range of degrees. To show this, we introduce a graded vector space of labeled graphs tailored to the given ring presentation. Let us therefore start by defining the notion of graphs that we will consider:

\begin{definition}[cf.\cite{KupersRandal-Williams}]
    Let $S=(S_1,S_2)$ be a pair of finite sets.\begin{enumerate}[(1)]
        \item A marked directed oriented graphs with \textit{incoming legs} labeled by $S_1$ and \textit{outgoing legs} labeled by $S_2$, consists of\begin{enumerate}[(i)]
        \item a totally ordered finite set $V$ (consisting of vertices), two totally ordered finite sets $H_1$ and $H_2$ (consisting respectively of incoming and outgoing half-edges) and monotone functions $a_1:H_1\to V$ and $a_2:H_2\to V$, encoding that a half edge is incident to a vertex,
        \item a bipartite matching $m$ of $(S_1\sqcup H_2, S_2\sqcup H_1)$, which encodes the edges and legs of the graph (and how the legs are labeled).
    \end{enumerate}
    We illustrate some graphs like this in Figure \ref{fig:2-1-valent-graphs}.
    \item For $v\in V$, we call the pair $(|a_1^{-1}(v)|,|a_2^{-1}(v)|)$, i.e.\ the pair of the number of incoming incident edges and the number of outgoing incident edges, the \textit{valency} of $v$. 
    \item  An isomorphism from a marked directed oriented graph $(V,H_1,H_2,a_1,a_2,m)$ to a marked directed oriented graph $(V',H_1',H_2',a_1',a_2',m')$ is a triple 
    $$(f:V\to V',g_1:H_1\to H_1,g_2:H_2\to H_2)$$
    of order preserving bijections such that $f\circ a_1=a_1'\circ g_1$, $f\circ a_2=a_2'\circ g_2$ and such that $(\id_{S_1}\sqcup g_2,\id_{S_2}\sqcup g_1)$ sends $m$ to $m'$. 
    \item An oriented directed graph is an isomorphism class of marked oriented graphs. We write $[\Gamma]$ for the isomorphism class of a marked directed oriented graph $\Gamma$.
    \end{enumerate} 
\end{definition}

\begin{figure}[h]
    \centering
    \includegraphics[scale=0.3]{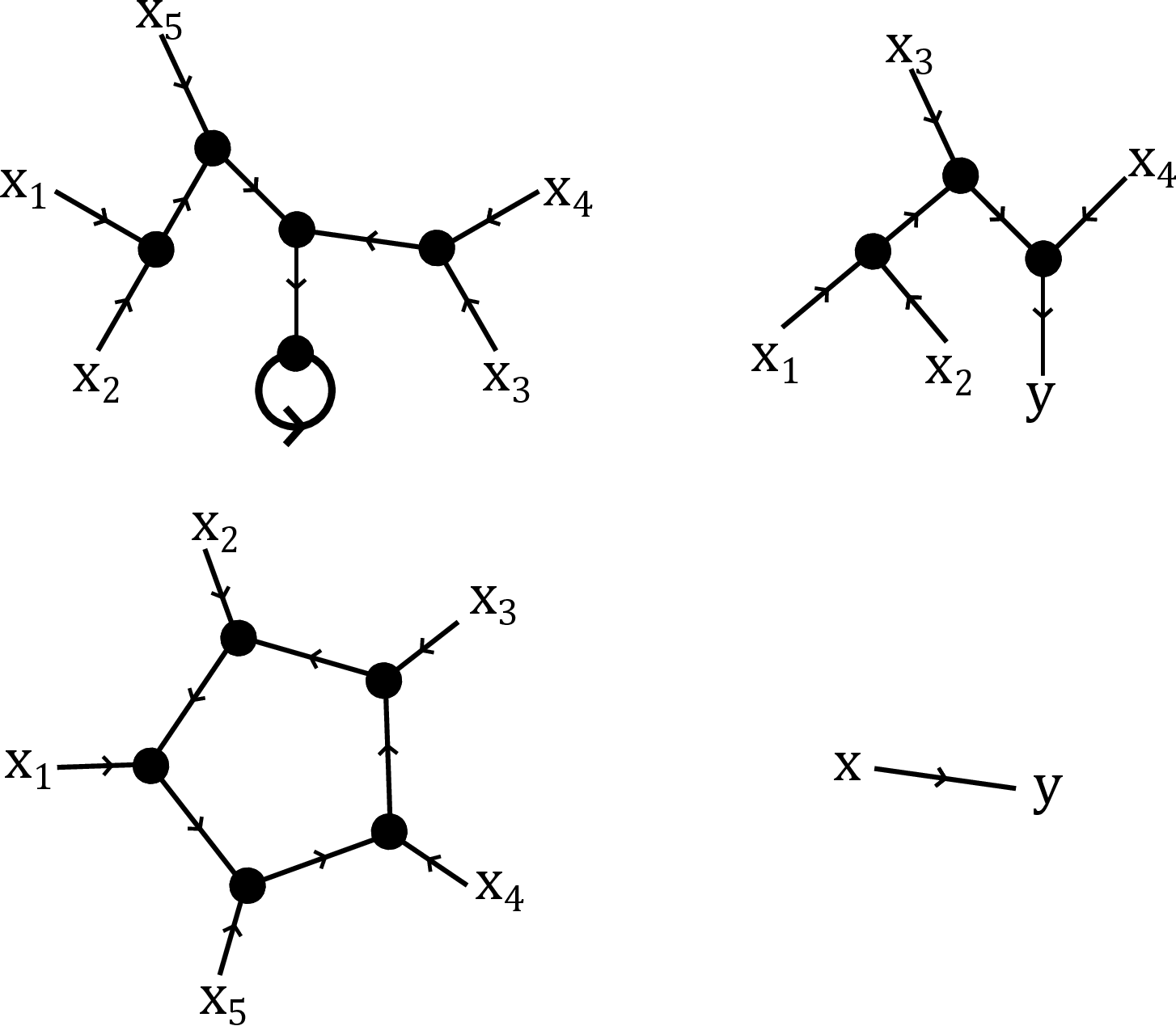}
    \caption{Some examples of $(2,1)$-valent direct graphs.}
    \label{fig:2-1-valent-graphs}
\end{figure}

\begin{remark}
    We remark that in this definition, we allow the matching $m$ to contain pairs of the form $(s_1,s_2)$, with $s_1\in S_1$ and $s_2\in S_2$. We illustrate this graphically by an oriented edge from the label $s_1$ to the label $s_2$ (as seen in Figure \ref{fig:2-1-valent-graphs}), although this oriented edge does not come from the half edges in $H_1$ and $H_2$.
\end{remark}

For $S=(S_1,S_2)$ a pair of finite sets, we let $\cC'(S)$ be the graded vector space, concentrated in degree $|S_1|-|S_2|$, with basis given by oriented graphs where each vertex has valency $(2,1)$, i.e.\ exactly one incident outgoing edge and two incident incoming edges. 

\begin{convention}
    In pictures, we will use the convention that incoming half-edges at a vertex are ordered clockwise, starting from the outgoing half-edge. Furthermore, we use the convention that vertices are ordered left to right and up to down.
\end{convention}

\begin{remark}
    It is useful to think of a vertex of valency $(2,1)$ of such a graph as corresponding to the cohomology class $h_{2,1}$ and a oriented edge between two labels in $(S_1,S_2)$ as corresponding to the class $h_{1,1}$.
\end{remark}

\begin{remark}
    Comparing our grading to the grading given on graphs in \cite{KupersRandal-Williams}, we could have used a similar convention to there by giving each incoming half-edge degree 1 and each outgoing half-edge degree $-1$, but since this means that any internal edge contributed to the degree by 0, this is the same as our grading.
\end{remark}

We also want to identify graphs that are related by reordering half-edges and vertices, but up to a sign. This has several reasons, as we will see, but the first is to be compatible with the relation (\ref{eq:rel2-Rpres}). Suppose $\Gamma=(V,H_1,H_2,a_1,a_2,m)$ and $\Gamma'=(V',H_1',H_2',a_1',a_2',m')$ are marked oriented graphs, representing basis elements of $\cC'(S)$, and $f:V\to V'$, $g_1:H_1\to H_1'$ and $g_2:H_2'\to H_2'$ are not necessarily order preserving bijections such that $f\circ a_1=a_1'\circ g_1$, $f\circ a_2=a_2'\circ g_2$ and $(\id_{S_1}\sqcup g_2,\id_{S_2}\sqcup g_1)$ sends $m$ to $m'$. 

First of all, since the set of vertices $V$ is totally ordered, the bijection $f$ has a well-defined sign, which denote by $\sgn(f)$.

Furthermore, for each $v\in V$, the bijection $g_1$ restricts to a bijection from the set $a_1^{-1}(v)$ and the set $(a_1')^{-1}(f(v))$, which are sets with two elements. Once again, this is a bijection between totally ordered sets, so there is a well-defined sign of the restriction, which we denote by $\sgn(g_1;v)$. We define the sign of $g_1$ to be
$$\sgn(g_1):=\prod_{v\in V}\sgn(g_1,v).$$
With this we make the following definition:

\begin{definition}
    For $S=(S_1,S_2)$ a pair of finite sets, we let $\cC(S)$ be the quotient of $\cC'(S)$ by the subspace generated by differences
    $$[\Gamma]-\sgn(f)\sgn(g_1)[\Gamma'],$$
    where $\Gamma$ and $\Gamma'$ are graphs related by bijections as in the previous two paragraphs.
\end{definition}

Finally, we want to impose a relation on graph corresponding to the relation (\ref{eq:rel3-Rpres}). Graphically, this corresponds to a kind of ``directed IH-relation'' (cf. relation ($\epsilon$) in \cite[Section 5]{KupersRandal-Williams}), which is illustrated in Figure \ref{fig:IHrel}. We make the following definition:

\begin{figure}[h]
    \centering
    \includegraphics[scale=0.3]{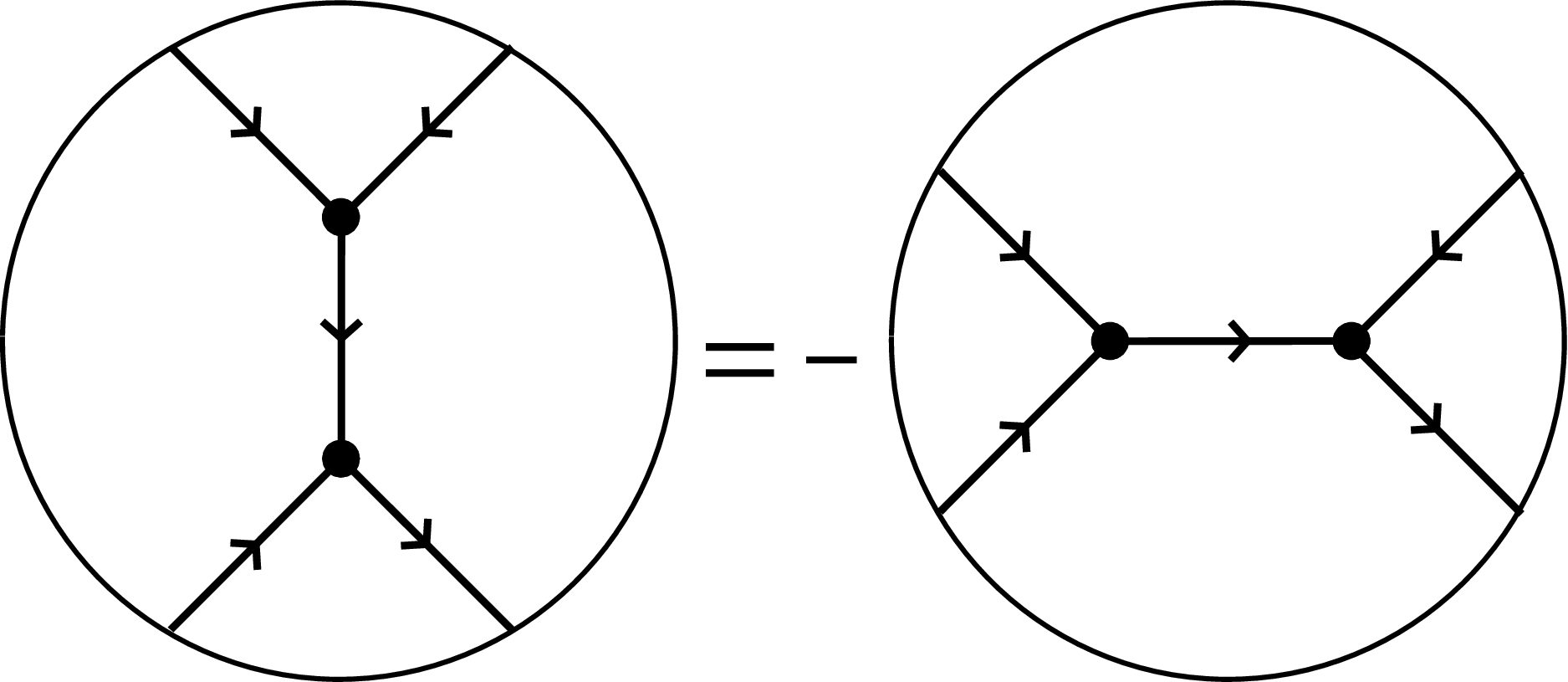}
    \caption{The directed IH-relation among $(2,1)$-valent directed graphs}
    \label{fig:IHrel}
\end{figure}

\begin{definition}
    We define $\cG(S)$ as the quotient of $\cC(S)$ by the subspace spanned by differences $[\Gamma]-[\Gamma']$,    where $\Gamma$ and $\Gamma'$ are related by the local move in Figure \ref{fig:IHrel}.
\end{definition}

Before we move on to prove the proposition, let us prove some elementary facts about the space $\cG(S)$:

\begin{proposition}\label{prop:graphreduction}
    Any graph in $\cG(S)$ is equivalent to one whose connected components are the kinds of graphs in Figure \ref{fig:stdgraphs}. 
\end{proposition}

\begin{figure}[h]
    \centering
    \includegraphics[scale=0.3]{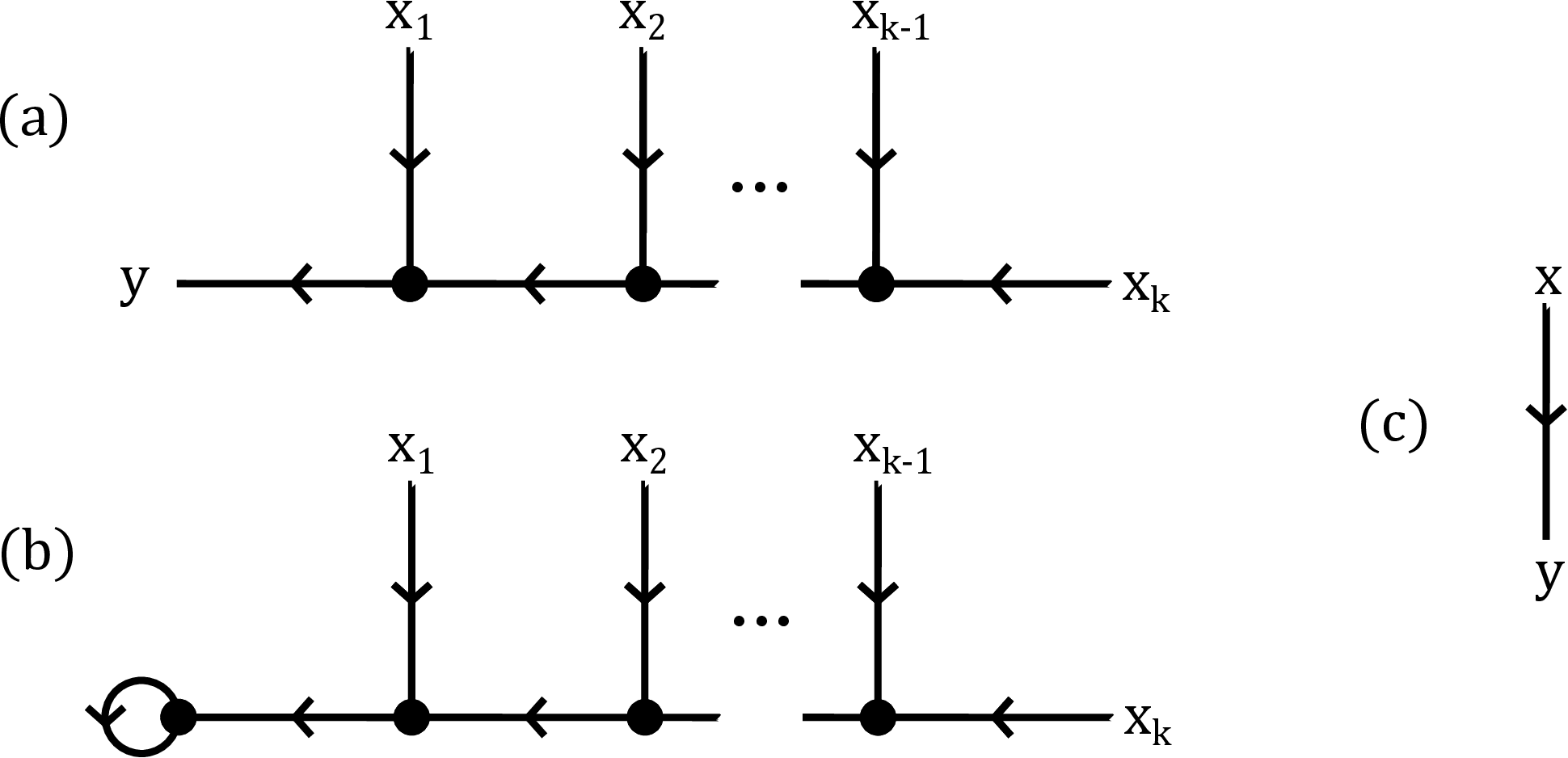}
    \caption{}
    \label{fig:stdgraphs}
\end{figure}

\begin{proof}
    First, we note that if a connected component of a graph contains no vertex, it must be of the form (c) in Figure \ref{fig:stdgraphs}.

    Next, suppose that a graph has a connected component which contains a vertex but which contains no directed cycle. Since no vertex can have two outgoing incident half-edges, the graph cannot have any non-directed cycles and must therefore be a directed, binary tree. By iteratively applying the move in Figure \ref{fig:IHrel}, any such tree can be made into the graph (a) in Figure \ref{fig:stdgraphs}.
    
    Finally, we  show that any graph containing a directed cycle is equivalent to a graph of the form (b) in Figure \ref{fig:stdgraphs}. If the graph has only one vertex, the cycle must be a loop and the vertex must have one additional incident incoming half-edge. Since the graph has no more vertices, it must be of the form (b), with one incoming labeled half-edge. Suppose next that the graph has more than one vertex, and contains a directed cycle. In this case, any directed cycle containing more than one vertex can be made shorter using the relation \ref{fig:IHrel}. Since the graphs are finite we can repeat this procedure a finite number of times until the cycle contains only one vertex and thus one edge, i.e.\ is a loop. Once again, this means that there is precisely one incoming half-edge to the vertex which is not in the loop, so the loop can not be part of another directed cycle. This means that if there are other directed cycles in the graph, we can repeat the previous step a finite number of times, until we have a directed graph which has no directed cycles apart from possible loops at some vertices. However, if we remove the loops from such a graph, we get a directed binary tree where each vertex either has two incoming and one outgoing half-edge, or only one incoming half-edge, and such a directed tree can only have one vertex with only one incoming half-edge, so the graph can only have had one loop. Finally, we may once again apply the local move from Figure \ref{fig:IHrel} a number of times to get a graph of the form (b) in Figure \ref{fig:stdgraphs}. This finishes the proof.
\end{proof}

With this, we are ready to prove the proposition.

\begin{proof}[Proof of Proposition \ref{prop:ringstructure}]
    We follow closely the proofs of \cite[Theorems 5.1 and 5.2]{KupersRandal-Williams}. Recall that we want to show that for any pair of finite sets $S=(S_1,S_2)$, the map
    $$\phi_*:[K_S(n)\otimes R_{\pres}]^{\GL_n(\bZ)}\to[K_S(n)\otimes R]^{\GL_n(\bZ)} ,$$
induced by the map $R_\pres\to R$, is surjective and in a stable range an isomorphism in each degree. For brevity, let us write $G:=\GL_n(\bZ)$ for the remainder of the proof.

    Let us write $R_\gen$ for the graded commutative ring generated by the classes $\kappa_1(f_1\otimes f_2\otimes v)$, modulo linearity in $f_1, f_2$ and $v$. In other words, we have
    $$R_\gen:=\exterior{*}K_{1,2}(n).$$
    Given a graph $[\Gamma]=(V,H_1,H_2,a_1,a_2,m)\in\cC'(S)$, we have a map
    \begin{align}\label{eq:graphmap}
        \alpha_\Gamma:\bigotimes_{v\in V}\kappa_1\otimes\id:\bigotimes_{v\in V} H^{\vee}(n)^{\otimes a_1^{-1}(v)}\otimes H(n)^{\otimes a_2^{-1}(v)}\otimes K_S(n)\to R_{\gen}\otimes K_S(n) 
    \end{align}
    There is an invariant element
    $$\omega_m(1)\in\bigotimes_{v\in V} H^{\vee}(n)^{\otimes a_1^{-1}(v)}\otimes H(n)^{\otimes a_2^{-1}(v)}\otimes K_S(n),$$
    and so we can define a map $\alpha':\cC'(S)\to[R_{\gen}\otimes K_S(n)]^G$ by evaluating the map (\ref{eq:graphmap}) on $\omega_m(1)$, i.e.\ $\alpha'([\Gamma])=\alpha_\Gamma(\omega_m(1))$. By Proposition \ref{prop:invariantsofK}, this map is surjective.
    
    The relations defining the quotient $\cG(S)$ correspond to the relations in $R_{\pres}$, as we have seen, so it follows that these are in the kernel of the composition of $\alpha'$ with the projection $[K_S(n)\otimes R_\gen]^G\to[K_S(n)\otimes R_\pres]^G$ and thus we get a map $\alpha:\cG(S)\to[K_S(n)\otimes R_{\pres}]^G$, which is also surjective.

    By Proposition \ref{prop:wBr2}, there is a map
    $$\psi^{\wBr_n}_S:\PP(S)\otimes\det(S)\to [K_S(n)\otimes R]^{G},$$
    which is an epimorphism and an isomorphism in a stable range of degrees. We want to relate this map to the composition $\phi_*\circ\alpha:\cG(S)\to[K_S(n)\otimes R]^G$

    By Proposition \ref{prop:graphreduction}, every graph in $\cG(S)$ is equivalent to a graph whose connected components are of the forms in Figure \ref{fig:stdgraphs}. Using this description, we can define a map $\cG(S)\to \PP(S)\otimes\det(S)$, by sending each component of the form (c) to a labeled part $(s,t)$, the components of the form (a) to the labeled part $(\{s_1,\ldots,s_k\},t')$ and the components of the form (b) to the unlabeled part $\{s_1,\ldots,s_k\}$ and tensoring everything with the volume form in $\det(S)$. 
    
    From this description, it is clear that this map is surjective, since any labeled partition can be realized by a graph with components as in Proposition \ref{prop:graphreduction}. Furthermore, reducing graphs to this form shows that non-equivalent graphs are mapped to different labeled partitions and thus the map is in fact also injective. 
    
    Comparing the definition of $\Psi_S^{\wBr_n}$ with that of $\alpha$, we see that the diagram
    \[\begin{tikzcd}
        \cG(S)\arrow[r,"\alpha"]\arrow[d]& \left[R_{\pres}\otimes K_S(n)\right]^G\arrow[d,"\phi_*"]\\
        \PP(S)\otimes\det(S)\arrow[r,"\Psi_S^{\wBr_n}"]&\left[R\otimes K_S(n)\right]^G
    \end{tikzcd}\]
    is commutative. Since $\cG(S)\to\PP(S)\otimes\det(S)$ is an isomorphism and $\Psi_S^{\wBr_n}$ is surjective and an isomorphism in a stable range of degrees, $\phi_*$ must be surjective and an isomorphism in a stable range of degrees as well.
\end{proof}

\appendix 
\section{The stable Albanese (co)homology of \texorpdfstring{$\IA_n$}{IA}}\label{Appendix}
\vspace{-5pt}
\begin{center}
    \sectionauthor{By Mai Katada}
\end{center}

In this appendix, we will study the $\GL_n(\bQ)$-representation structure of the Albanese homology of $\IA_n$, which is the dual of the Albanese cohomology of $\IA_n$.


We will recall our conjectural structure of the stable Albanese homology of $\IA_n$.
For $i\ge 1$, let $U_i=\Hom(H,\bigwedge^{i+1}H)$, where $H=H(n)=H_1(F_n,\bQ)$.
Let $U_*=\bigoplus_{i\ge 1}U_i$ be the graded $\GL_n(\bQ)$-representation.
Define $W_*=W_*(n)=\widetilde{S}^*(U_*)$ as the \emph{traceless part} of the graded-symmetric algebra $S^*(U_*)$ of $U_*$.
Here, the traceless part consists of elements that vanish under any contraction maps between distinct factors of $S^*(U_*)$.
See \cite[Sections 2.5 and 2.6]{Katada2022stable} for details.
We can reconstruct $W_*$ by using a non-unital wheeled PROP $\calC_{\calO^{\circlearrowright}}$ as in Lemma \ref{lemmaW*} below.

The goal of this appendix is to prove the following theorem.

\begin{theorem}[Cf.\;{\cite[Conjecture 6.2]{Katada2022stable}}]\label{TheoremAlbaneseIAn}
    We have an isomorphism of graded $\GL_n(\bQ)$-repre- sentations
    \begin{gather*}
        H_*^A(\IA_n,\bQ)\cong W_*(n)
    \end{gather*}
    for sufficiently large $n$ with respect to the homological degree $*$.
\end{theorem}

Note that the statement of \cite[Conjecture 6.2]{Katada2022stable} is for $n\ge 3*$.
In a subsequent paper \cite{KatadaRange}, we will prove Theorem \ref{TheoremAlbaneseIAn} with this stable range and also prove several related conjectures proposed in \cite{Katada2022stable}.

The case of degree $1$ follows from Kawazumi \cite{KawazumiMagnusExpansions}, Cohen--Pakianathan and Farb (both unpublished), and the cases of degree $2$ and $3$ are proven by Pettet \cite{Pettet} and the author of this appendix \cite{Katada2022stable}, respectively.
Moreover, in general degrees, a partial result is given in \cite[Theorem 6.1]{Katada2022stable}.

\begin{theorem}[{\cite[Theorem 6.1]{Katada2022stable}}]\label{TheoremAlbaneseinclusion}
    We have a morphism of graded $\GL_n(\bQ)$-representations
    \begin{gather*}
        F_*: H_*(U_1,\bQ)\to S^*(U_*)
    \end{gather*}
    such that $F_*(H^A_*(\IA_n,\bQ))\supset W_*$ for $n\ge 3*$.
\end{theorem}

In order to prove Theorem \ref{TheoremAlbaneseIAn}, we will use \cite[Proposition 12.3]{Katada2022stable}. To state this proposition, we introduce the wheeled PROP and the non-unital wheeled PROP corresponding to the operad $\Com$. 

Let $\calP_0=\bigoplus_{k\ge 1}\calP_0(k)$ denote the operadic suspension of the operad $\Com$ of non-unital commutative algebras.
Thus, we have $\calP_0(0)=0$ and $\calP_0(k)=\det(\bQ^{k})[k-1]$ in cohomological dimension $k-1$ for $k\ge 1$.
Let $\calP_0^{\circlearrowright}$ denote the wheeled completion of $\calP_0$ and $\calC_{\calP_0^{\circlearrowright}}$ the wheeled PROP freely generated by $\calP_0^{\circlearrowright}$, which was introduced by \cite{KawazumiVespa}.
Let $\calO=\bigoplus_{k\ge 2}\calP_0(k)$ denote the non-unital suboperad of $\calP_0$, $\calO^{\circlearrowright}$ the non-unital wheeled sub-operad of $\calP_0^{\circlearrowright}$ and $\calC_{\calO^{\circlearrowright}}$ the non-unital wheeled sub-PROP of $\calC_{\calP_0^{\circlearrowright}}$.
By these constructions, the wheeled PROP $\calC_{\calP_0^{\circlearrowright}}$ corresponds to the $\wBr_n$-module $\calP\otimes \det$ and the non-unital wheeled PROP $\calC_{\calO^{\circlearrowright}}$ corresponds to the $\dwBr$-module $\calP'\otimes \det$.

\begin{lemma}[{\cite[Proposition 12.3]{Katada2022stable}}]\label{lemmaW*}
    We have an isomorphism of $\GL_n(\bQ)$-representations
    \begin{gather*}
        W_i(n)\cong \bigoplus_{p-q=i}K^{\circ}_{p,q}(n)\otimes_{\bQ[\Sigma_p\times \Sigma_q]} \calC_{\calO^{\circlearrowright}}(p,q)
    \end{gather*}
    for sufficiently large $n$ with respect to $i$.
\end{lemma}

\begin{proof}[Proof of Theorem \ref{TheoremAlbaneseIAn}]

By Proposition 6.3, we stably have an isomorphism 
\begin{gather*}
    R\cong R_{\text{pres}},
\end{gather*}
of $\GL_n(\bQ)$-representations. By Remark 6.2, we have
\begin{gather*}
    R_{\text{pres}}\cong H^*(U_1,\bQ)/{\langle R_2 \rangle},
\end{gather*}
where $\langle R_2 \rangle$ denotes the two-sided ideal of $H^*(U_1,\bQ)$ generated by $$R_2=\ker(H^2(U_1,\bQ)\cong {\bigwedge}^2 H^1(\IA_n,\bQ) \xrightarrow{\cup} H^2(\IA_n,\bQ)).$$
Since the Albanese cohomology of $\IA_n$ is defined to be the image of the cup product map 
$$\bigwedge^* H^1(\IA_n,\bQ) \xrightarrow{\cup} H^*(\IA_n,\bQ),$$
we have a surjective map of $\GL_n(\bQ)$-representations
\begin{gather*}
    H^*(U_1,\bQ)/{\langle R_2 \rangle}\twoheadrightarrow H^*_A(\IA_n,\bQ).
\end{gather*}
Therefore, we stably have a surjective map of $\GL_n(\bQ)$-representations
\begin{gather*}
    R\cong R_{\text{pres}}\cong  H^*(U_1,\bQ)/{\langle R_2 \rangle}\twoheadrightarrow  H^*_A(\IA_n,\bQ).
\end{gather*}
We stably have 
$$R_i\cong W_i(n)^{\vee}$$
by Lemma \ref{lemmaW*}
since we stably have an isomorphism of $\GL_n(\bQ)$-representations
\begin{gather*}
\begin{split}
    R_i=(i^*(K^{\vee})\otimes^{\dwBr}(\calP'\otimes \det))_i
    &\cong \bigoplus_{p-q=i} K^{\circ}_{q,p}(n)\otimes_{\bQ[\Sigma_p\times \Sigma_q]}(\calP'\otimes \det)(p,q)\\
    &\cong 
    \bigoplus_{p-q=i} K^{\circ}_{q,p}(n)\otimes_{\bQ[\Sigma_p\times \Sigma_q]}\calC_{\calO^{\circlearrowright}}(p,q),
\end{split}
\end{gather*}
where the first isomorphism is verified by Proposition 2.7.
Therefore, the statement follows from Theorem \ref{TheoremAlbaneseinclusion}.
\end{proof}

\printbibliography

\vspace{1em}

\noindent {Erik Lindell, \itshape IMJ-PRG, 75013 Paris, France.} \noindent {Email address: \tt erikjlindell@gmail.com, \\ lindell@imj-prg.fr}

\vspace{0.8em}

\noindent{Mai Katada, \itshape Faculty of Mathematics, Kyushu University, 744, Motooka, Nishi-ku, Fukuoka, Japan
819-0395.} \noindent {Email address: \tt katada@math.kyushu-u.ac.jp}

\end{document}